\documentclass[11pt]{amsart}

\usepackage{color,graphicx,enumerate,wrapfig,amssymb}

\usepackage{subfigure, caption}

\usepackage{eucal}

\usepackage{setspace}

\usepackage{hyperref}

\usepackage{algorithm}
\usepackage[noend]{algpseudocode}

\makeatletter
\def\BState{\State\hskip-\ALG@thistlm}
\makeatother

\usepackage{graphicx}
\usepackage{transparent}

\usepackage{enumitem}

\newcommand\numberthis{\addtocounter{equation}{1}\tag{\theequation}} 

\hypersetup{
    bookmarks=true,         
    unicode=true,          
    pdftoolbar=true,        
    pdfmenubar=true,        
    pdffitwindow=false,     
    pdfstartview={FitH},    
    pdftitle={My title},    
    pdfauthor={Author},     
    pdfsubject={Subject},   
    pdfcreator={Creator},   
    pdfproducer={Producer}, 
    pdfkeywords={keyword1} {key2} {key3}, 
    pdfnewwindow=true,      
    colorlinks=true,       
    linkcolor=blue,          
    citecolor=blue,        
    filecolor=magenta,      
    urlcolor=cyan           
}

\renewenvironment{proof}[1][\proofname ]{{\noindent \bfseries #1. }}{\qed \bigskip } 
\def\subjclass#1{{\renewcommand{\thefootnote}{}%
\footnote{\emph{Mathematics Subject Classification (2010):} #1}}}

\newcommand{\R}{{\mathbb R}}

\newcommand{\Z}{{\mathbb Z}}
\newcommand{\N}{{\mathbb N}}

\newcommand{\e}{\varepsilon}

\newcommand{\supp}{\operatorname{supp}}

\newcommand{\CAP}{\kappa_0}  
\newcommand\INT[1]{\overset{\circ}{ #1}}  

\floatname{algorithm}{Algorithm}

\newtheorem{theorem}{Theorem}[section]

\newtheorem{cor}[theorem]{Corollary}

\newtheorem{definition}{Definition}[section]

\newtheorem{lem}[theorem]{Lemma}

\newtheorem{prop}[theorem]{Proposition}

\newtheorem{remark}[theorem]{Remark}

\makeatletter
\def\@tocline#1#2#3#4#5#6#7{\relax
  \ifnum #1>\c@tocdepth 
  \else
    \par \addpenalty\@secpenalty\addvspace{#2}%
    \begingroup \hyphenpenalty\@M
    \@ifempty{#4}{%
      \@tempdima\csname r@tocindent\number#1\endcsname\relax
    }{%
      \@tempdima#4\relax
    }%
    \parindent\z@ \leftskip#3\relax \advance\leftskip\@tempdima\relax
    \rightskip\@pnumwidth plus4em \parfillskip-\@pnumwidth
    #5\leavevmode\hskip-\@tempdima
      \ifcase #1
       \or\or \hskip 2em \or \hskip 2em \else \hskip 3em \fi%
      #6\nobreak\relax
    \dotfill\hbox to\@pnumwidth{\@tocpagenum{#7}}\par
    \nobreak
    \endgroup
  \fi}
\makeatother

\numberwithin{equation}{section}

\title[Discrete Balayage and Boundary Sandpile]{Discrete Balayage and Boundary Sandpile}

\author{Hayk Aleksanyan}
\address{Department of Mathematics, KTH Royal Institute of Technology, SE-100 44  Stockholm,
Sweden}
\email{hayk.aleksanyan@gmail.com}

\author{Henrik Shahgholian}
\address{Department of Mathematics, KTH Royal Institute of Technology, SE-100 44  Stockholm,
Sweden}
\email{henriksh@math.kth.se}

\thanks{H. A. was supported by postdoctoral fellowship from Knut and Alice Wallenberg Foundation. 
H. Sh. was partially supported by Swedish Research Council.}


\begin{document}


\begin{abstract}    
We introduce a new lattice growth model, which we call boundary sandpile. The model   amounts to potential-theoretic redistribution of a given initial mass on $\Z^d$
($d\geq 2$) onto the boundary of an (a priori) unknown domain. The latter  evolves through sandpile dynamics, 
and has the property that the mass on the boundary is forced to stay below a prescribed threshold. 
Since finding the domain is part of the problem, the redistribution process is a discrete model of  a free boundary problem, whose continuum limit is yet 
to be understood.

We prove general results concerning our model. These include  canonical representation of the model  in terms of the smallest
super-solution among a certain class of functions,  uniform Lipschitz regularity of the scaled odometer function, and hence the
convergence of a subsequence of the  odometer  and the visited sites, discrete symmetry properties, as well as directional
monotonicity of the odometer function. The latter (in part)  implies the Lipschitz regularity of the free boundary of the sandpile. 

As a direct application of some of the methods developed in this paper,
combined with earlier results on classical Abelian sandpile,
we show that the boundary of the scaling limit of Abelian sandpile
is locally a Lipschitz graph.
\end{abstract}

\maketitle

{\small{ \tableofcontents}}

\section{Introduction}

\subsection{Background}
Recent years have seen a surge of modelling particle dynamics 
from a discrete  potential theoretic perspective. The models, which usually run under the heading
{\it aggregation models}, in particular cases boil down to (harmonic/Poisson) redistribution of a given initial mass  (sandpile) according to some prescribed governing rules. The most famous and well-known model is the
  Poincar\'e's Balayage, where a given initial mass distribution $\mu$ (in a given domain $D$) is to be redistributed (or mapped) to another distribution $\nu$ on the boundary of the domain
  \begin{equation}\label{balayage}
\mathrm{Bal}:  \mu   \longrightarrow    \nu ,
\end{equation}
where $\nu$ is uniquely defined through $\mu$, and $D$.
This model uses continuous amounts of mass instead of discrete ones; the latter  is more common in  chip firing on graphs
(see \cite{BLS} for instance).

 A completely different model called \emph{partial-balayage} (see \cite{Gust-part})\footnote{The term partial-balayage was first coined by Ognyan Kounchev as was pointed to us by Bj\"orn Gustafsson at KTH.} 
 aims at finding a body (domain) that is gravi-equivalent  with the given initial mass. This problem, in turn, is  equivalent to variational inequalities and the so-called obstacle problem.
 The discrete version of this problem was (probably for the first time) studied  by 
 D. Zidarov, where he performed (what is now called) a \emph{divisible  sandpile} model; see  \cite{Zidarov} page 108-118.\footnote{Actually Zidarov (at page 109 in his book)
 claims to  prove that the model is abelian. We have not attempted to verify the correctness of his claim.}
 
Levine \cite{Lev-thesis}, and Levine-Peres \cite{Lev-Per}, \cite{Lev-Per10} started a systematic study of such problems, 
proving, among other things, existence of scaling limit for divisible sandpiles.  Although Zidarov was the first to consider such a problem,
the mathematical rigour is to be attributed to Levine \cite{Lev-thesis}, and Levine-Peres \cite{Lev-Per}, \cite{Lev-Per10}.

The divisible sandpile, which is of particular relevance to our paper,
is a growth model on $\Z^d$ $(d\geq 2)$ which amounts to redistribution of
a given continuous mass. 
The redistribution of mass takes place according to a (given) simple rule:  each lattice point can, and eventually must topple if it already has more than
a prescribed amount of mass (sand). The amount to topple is the excess, which is divided between all the neighbouring lattice points equally or according to a governing background PDE.
The scaling limit of this model, when the lattice spacing tends to 0, and the amount of mass
is scaled properly, leads to the \emph{obstacle problem} in $\R^d$ $(d\geq 2)$.

The divisible sandpile model of Zidarov, and Levine-Peres also relates to a well-known problem in potential theory, the so-called Quadrature Domains (QD) \cite{QD-book}.
A quadrature domain $D$ (with respect to a given measure $\mu$) is a domain that has the same exterior Newtonian  potential (with uniform density)
as that of the measure $\mu$. Hence, potential theoretically $\mu$ and $\chi_D$ are equivalent in the free space; i.e. outside the support of $D$
one has
$U^\mu = U^{\chi_D}$, where these are the Newtonian potentials of $\mu$, respectively $\chi_D$.
The odometer function  $u$ of Levine-Peres (which represents the amount of mass emitted from each lattice site) corresponds to the difference between the above potentials (up to a normalization constant),
i.e. $c_d u (x) = U^\mu - U^{\chi_D}$, where $u$ is $C^{1,\alpha} (\R^d)$ and 
$u=|\nabla u| = 0$ in $\R^d \setminus D$. This, expressed differently, means that
\begin{equation}\label{QD}
\int h(x) d\mu = \int_D h(x) dx , 
\end{equation}
for all $h$ harmonic and integrable over $D$ (see \cite{QD-book}).

In many other (related) free boundary value problems (known as Bernoulli type) the zero boundary gradient  $|\nabla u | =0$ in the above takes a different turn, and is a prescribed (strictly) non-zero function, 
and the volume potential $ U^{\chi_D}$ in \eqref{QD} is replaced by  surface/single layer potential (of the a priori unknown domain) $ U^{\chi_{\partial D}}$. 
In terms of sandpile redistribution this means to find a domain $D$ such that the given  initial mass $\mu$ in \eqref{balayage} is replaced by a prescribed 
mass $\nu = g(x) d \mathcal{H}_{\partial D} $ on $\partial D$.
Here $ \mathcal{H}_{\partial D}$ is the standard surface measure on $\partial D$, and $g$ a given function;
for the simplest case  $g(x) \equiv 1$ the counterpart of \eqref{QD} would be
 \begin{equation}\label{QS}
\int h(x) d\mu = \int_{\partial D} h(x) d \mathcal{H}_{\partial D},
\end{equation}
for all $h$ harmonic on a neighbourhood of $\overline D$.  
Such domains are referred to as Quadrature Surfaces, with corresponding PDE 
$$
\Delta u = -\mu + d \mathcal{H}_{\partial D}
$$
with $c_d u= U^\mu -U^{\chi_{\partial D}}$, where $c_d$ is a normalization constant.
This problem is well-studied in the continuum case and 
there is a large amount of literature concerning existence, regularity and geometric properties of both solutions and the free boundary, 
see \cite{AC}, \cite{GS}, and the references therein.
 
In our search to find a model for the sandpile particle dynamics to model \eqref{QS} we came across a few models that we (naively) thought could and should be the answer to our quest. However, running numerical
simulations (without any theoretical attempts)
it became apparent that the models we suggested are far from the Bernoulli
type free boundaries\footnote{
Although uniqueness (in general) fails  for Bernoulli type free boundary problems,
it is well-known that in the case of good geometry of data (here Dirac mass) and (partial) smoothness of the free boundary
one obtains uniqueness of the solution. In particular a domain $D$ admitting quadrature identity \eqref{QS},
 in continuum case, should   be a sphere.}. 
In particular, redistribution of  an initial  point mass does not converge (numerically) to a sphere, but
to a shape with a boundary remotely resembling the boundary of the Abelian sandpile
(see Figure\footnote{This of course is far from a  reasonable shape for a Bernoulli free boundary!} \ref{Fig-1}; cf. \cite[Figure 1]{PS}, or \cite[Figure 2]{Lev-Per} for instance).

\begin{figure}[!htbp]
\centering
\begin{subfigure}
  \centering
  \includegraphics[width=1.8in]{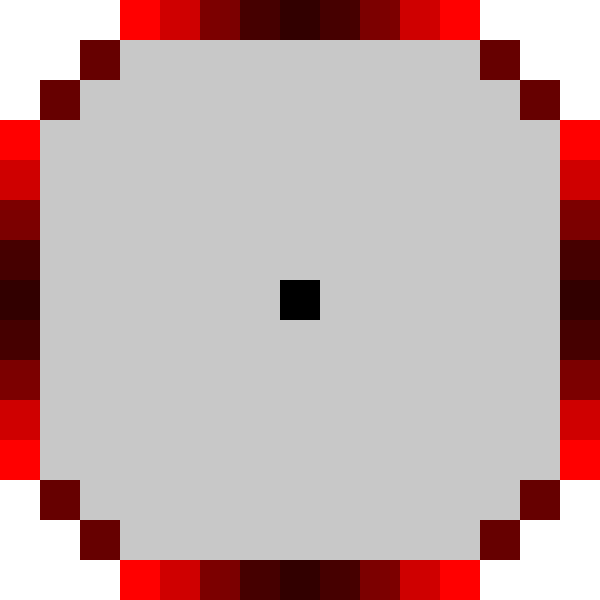}
  \label{fig:sub1}
\end{subfigure}%
\qquad
\begin{subfigure}
  \centering
  \includegraphics[width=1.8in]{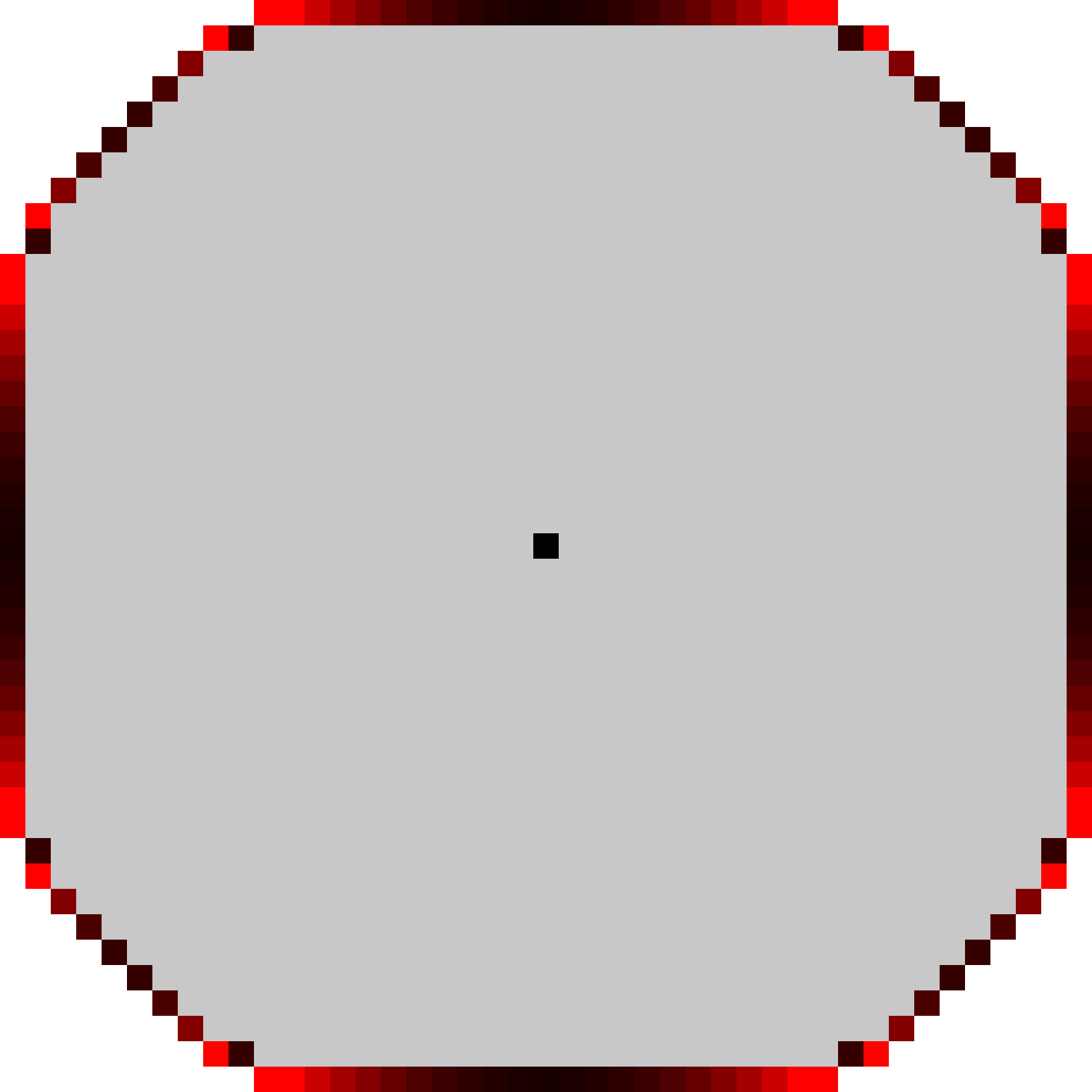}
  \label{fig:sub2}
\end{subfigure}

\vspace{0.5cm}

\begin{subfigure}
  \centering
  \includegraphics[width=1.8in]{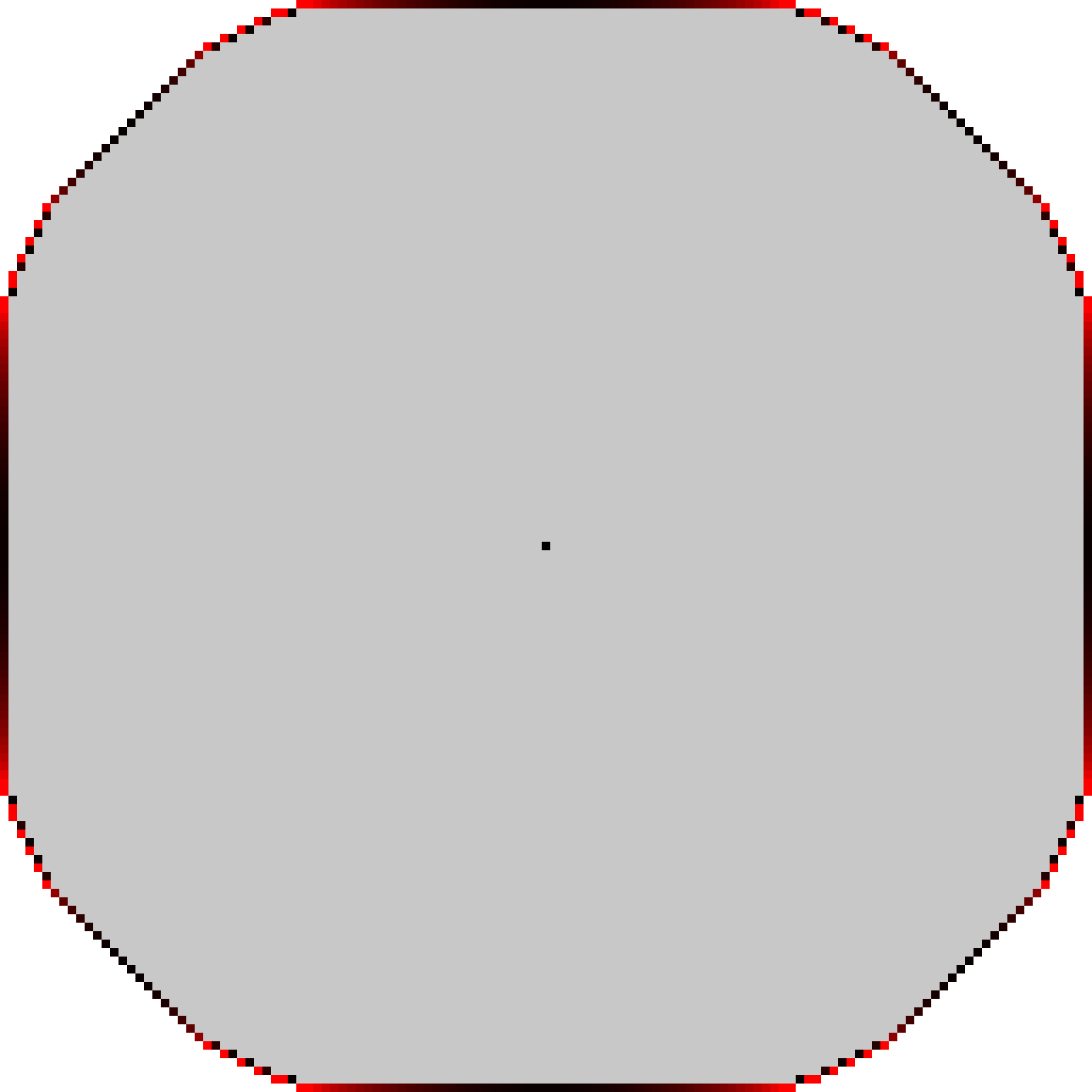}
  \label{fig:sub1}
\end{subfigure}%
\qquad
\begin{subfigure}
  \centering
  \includegraphics[width=1.8in]{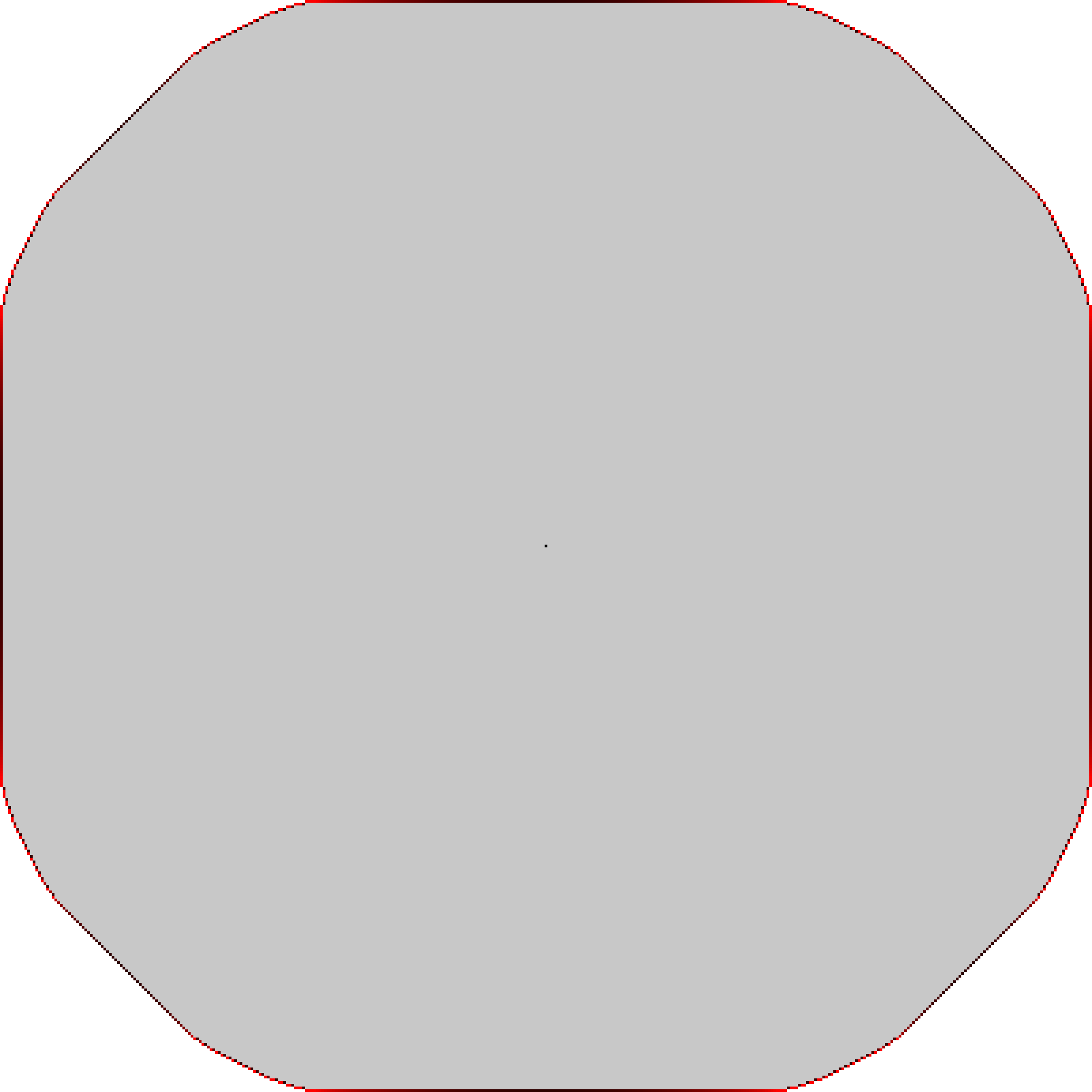}
  \label{fig:sub2}
\end{subfigure}
\caption{\footnotesize{From top left to bottom right are the sets of visited sites on $\Z^2$ of
 the 2-dimensional BS with initial mass $n$ concentrated at the origin and equal to
$10^3$, $10^4$, $10^5$ and $10^6$ respectively, where in all cases the boundary capacity is set to $n^{1/2}$.
The gray region is the set where the odometer is harmonic. On the boundary, darker colors indicate higher concentration of mass,
and the black dot at the center represents the origin.}}
\label{Fig-1}
\end{figure}

Notwithstanding this, our model seems to present a new alluring and fascinating  phenomenon not considered earlier,
neither by combinatorics nor free boundary communities. Hence the ``birth" of this article.

\subsection{Boundary sandpile model}
Since the boundary of a set will be a prime object in the analysis
of the paper, we will assume throughout the text that $d\geq 2$
to avoid uninteresting cases.

Given an initial mass $\mu$ on $\Z^d$, we want to find a model that builds a domain $D\subset \Z^d$ such that the corresponding (discrete) balayage measure $\nu$ on the boundary $\partial D$
(see \eqref{balayage}) is prescribed. Such a model seems for the moment infeasible for us.
A slight variation of it asks to find a ``canonical" domain such that the boundary mass $\nu$ stays below the prescribed threshold. Since larger
domains (with reasonable geometry)  have larger boundaries, and hence a greater amount of boundary points, we should expect that one can always find a solution to our problem by taking a very large domain
containing the initial mass. This, in particular, suggests that a canonical domain should be the smallest domain among all domains with the property 
that the boundary mass $\nu$ is below the prescribed mass, which we shall assume to be uniform mass.\footnote{Solving the above problem in $\R^d$ 
results in a so-called Bernoulli free boundary problem, where the Green's potential $G_{\mu,D}$ of the mass $\mu$, in the sought domain $D$, 
has boundary gradient $|\nabla G_{\mu,D}|= 1$, in case of uniform distribution, i.e. $\nu = d \mathcal{H}_{\partial D}$.} 

Departing from the above point of view of canonicity one may ask whether there is a natural
lattice growth model which corresponds to this 
minimality.\footnote{It should be remarked that the divisible sandpile related to the obstacle problem corresponds to smallest (super)solution to 
$\Delta w \leq 1$ with $w\geq  G_{-\mu}$.} 
It turns out that there actually is a model which seemingly is more complicated than the divisible sandpile model described above.
To introduce our model, which we  call  \emph{boundary sandpile},
we need a few formal definitions\footnote{Although this is a generic name, as there are obviously many
other ways to topple particles/mass onto the boundary (see Remark \ref{Rem-general-F}), we think it might be more convenient without any further specification with descriptive names
to use the term Boundary Sandpile.}.

\subsubsection{Constants}
Through the text various letters $c, C, c_1,...$ will stand for constants
which may change from formulae to formula. We use lower case $c$ to denote small constants,
and upper case $C$ for large constants. 

\subsubsection{Lattice notation}
Call two lattice points $x,y \in \Z^d$ {\it neighbours} and write $x\sim y$ if $||x-y||_{l^1} =1$,
where $l^1$-norm of $x\in \R^d$ is the sum of absolute values of its coordinates.
Clearly $x\sim y$ iff $x=y \pm e_i$ for some $1\leq i \leq d$ where $e_i$
is the $i$-th vector of the standard basis of $\R^d$.
For any non-empty set $V\subset \Z^d$, this concept of lattice adjacency
induces a natural graph $G$ having vertices on
$V$. We call such $G$ the \emph{lattice graph} of $V$.

Next, for a set $V\subset \Z^d$ define
$$
\partial V:= \{ x\in V: \exists \ y\sim x \text{ such that }  y\notin V \} \ \text{ and } 
\ \INT{V} = V \setminus \partial V,
$$
and call them respectively the \textbf{boundary} of $V$ and the \textbf{interior}
of $V$. For the convenience of the exposition we have chosen to include the boundary into the set.

For $h>0 $ recall the definition of the \emph{discrete Laplace operator}
denoted by $\Delta^h$
and acting on a function $u: h\Z^d \to \R$ by
$$
\Delta^h u(x) = \frac{1}{2d h^2} \sum\limits_{y\sim_h x} [u(y) - u(x)], \qquad x\in h\Z^d,
$$
where $y\sim_h x$, as for $h=1$, means that $y$ is a neighbour of $x$ in the lattice $h \Z^d$.
Throughout the text, when $h=1$, i.e. on a unit scale lattice, we will write $\Delta$ for $\Delta^1$
unless explicitly stated otherwise.

For $R>0$ we set
\begin{equation}\label{def-Z-R}
Z_R = \{\xi \in \Z^d: \ |\xi|\leq R\} \cup 
\{x\in \Z^d: \ \exists \ \xi \in \Z^d  \text{ s.t. } \xi \sim x \text{ and } |\xi |\leq R< |x|   \}
\end{equation}
for the closed discrete ball of radius $R$ including its 1-neighbourhood.
It is clear from the definition that $x\in \Z^d$ is from the interior of $Z_R$
iff $|x|\leq R$.

\subsubsection{Definition of the process}
Let now  $\mu_0$ be a given (mass) distribution on $\Z^d$, i.e. a non-negative function supported on finitely many sites of $\Z^d$. Fix also a threshold $\CAP>0$. For each integer $k\geq 0$ we inductively construct a sequence of sets $V_k$,
mass distributions $\mu_k$, and functions $u_k$ as follows.
Start with $V_0 = \mathrm{supp} \mu_0$ and $u_0=0$.
For an integer time $k\geq 0$ a particular site $x\in \Z^d$ is called \textbf{unstable}
if either of the following holds:
\begin{itemize}
 \item[(a)] $x\in \partial V_k$ and $\mu_k (x) > \CAP$,
 
 \item[(b)] $x\in \INT{V_k}$ and $\mu_k (x) >0$.
\end{itemize}
Otherwise a site is called \textbf{stable}.
We call the number $\CAP$ the \textbf{(boundary) capacity} of the model and refer to $V_k$ as the set of \emph{visited sites}
at time $k$.

Any unstable site can \emph{topple} by distributing all its mass equally among its $2d$ lattice neighbours.
More precisely, for each $k\geq 0$ we choose an unstable site $x\in V_k$ and define
$V_{k+1} = V_k \cup \{y\in \Z^d: \ y\sim x \}$,
\begin{equation}\label{mu-k+1-def}
 \mu_{k+1}(y) =  \begin{cases}  0, &\text{if $y=x$}, \\ 
 \mu_k(y) + \frac{1}{2d}\mu_k(x)  ,&\text{if  $ y\sim x    $}, \\
 \mu_k(y), &\text{otherwise},    \end{cases}
\end{equation}
and $u_{k+1} (y) = u_k(y) + \mu_k(y) \delta_0(y-x)$, $y\in \Z^d$,
where $\delta_0$ is the Kronecker delta symbol at the origin,
i.e. $\delta_0(x)$ equals 1 if $x=0$ and is zero otherwise for $x\in \Z^d$.
We call $u_k$ the \emph{odometer} function at time $k$.
For the sake of convenience, we do allow the toppling to be applied to a stable site,
as an identity operator, i.e. if at time $k$ a toppling is applied to a stable site $x$,
then we set $V_{k+1}= V_k$, $u_{k+1}=u_k$ and $\mu_{k+1}=\mu_k$.
We say that toppling $x$ is \emph{legal}, if $x$ is unstable.
If for some $k$ there are no unstable sites, the process is terminated.
We call this model \textbf{boundary sandpile} ($\mathrm{BS}$) and denote by $\mathrm{BS}(\mu_0, \CAP)$,
where $\mu_0$ is the initial distribution, and $\CAP$ is the boundary capacity of the model. 

It is clear that the triple $(V_k, \mu_k, u_k)_{k=1}^\infty$ may depend on the choice of the unstable sites,
i.e. the toppling sequence. 
Later on we will see that for a suitable class of toppling sequences
stable configurations exist and are identical (see Propositions \ref{Prop-stabil-toppling} and \ref{Prop-Abelian}).
Observe that from the definition of discrete Laplacian and (\ref{mu-k+1-def}) we easily see that
for each $k\geq 0$ one has 
\begin{equation}\label{Laplace-of-odometer}
 \Delta u_k (x) = \mu_k(x) - \mu_0(x) , \qquad x\in \Z^d,
\end{equation}
i.e. the Laplacian of $u_k$ represents the net gain of mass for a site $x$ at time $k$.

We further define a few concepts needed for our analysis. 
For a $\mathrm{BS}(\mu_0, \CAP)$, and a (toppling) sequence $T=\{x^{(k)}\}_{k=1}^\infty \subset \Z^d$,
we say that $T$ is \textbf{stabilizing}
if there exists a distribution $\mu$ such that
$\mu_k(x) \to \mu(x) $ as $k\to \infty$ for any $x\in \Z^d$.
We call $T$ \textbf{infinitive}
on a set $V\subset \Z^d$ if every $x\in V$ appears in the sequence $T$ infinitely often. 
If the set on which $T$ is infinitive is not specified, then it is assumed to be the entire $\Z^d$.

\subsection{Main results}
Our main results concern general qualitative  analysis of the boundary sandpile model, introduced in this paper. 
For any initial mass distribution we prove well-posedness of the model (Propositions \ref{Prop-stabil-toppling} and \ref{Prop-Abelian}),
and canonical representation of the model in terms of the smallest super-solution among a certain class of
functions (Theorem \ref{Thm-canonical}). 
Specifying our analysis for point masses and
using this minimality of the model, we show
directional monotonicity of the odometer function (Theorem \ref{Thm-monotonicity}).
We then determine the reasonable size of the boundary capacity
(subsection \ref{sub-sec-Heuristics}) and estimate the growth rate of the model in Lemma \ref{Lem-balls-inside-out}.
We prove a uniform Lipschitz bound for the scaled odometer function in Proposition \ref{Lem-Lipschitz-est}.
Using these results obtained in discrete setting, in the final part of the paper, Section \ref{sec-Shape-analysis}
we study the scaling limit of the model in continuum space.
In particular, we show that (along a subsequence) the scaled odometers, and hence the visited sites, converge to a continuum limit (Theorem \ref{Thm-scaling-limit} parts (i) and (iii)).
We also prove that the free boundary of (any) scaling limit of the model is locally a Lipschitz graph (Theorem \ref{Thm-scaling-limit} part (v)).

In subsection \ref{subsec-ASM} we apply our methods
developed for boundary sandpile model to classical Abelian sandpile.
Most notably, we show that the boundary of the scaling limit of Abelian sandpile
corresponding to initial configuration of chips at a single vertex, has Lipschitz boundary.

\begin{remark}\label{Rem-general-F}  
It is worthwhile to remark that the boundary sandpile described above can be represented slightly differently,
using discrete partial derivatives. Indeed, our model generates a set $V\subset \Z^d$ whose Green's function $u$
(corresponding to the given initial distribution of mass) satisfies
$$
\frac{1}{2d} \sum_{i=1}^d \left(|\partial_i^+  u   (x)| + |\partial_i^-  u   (x)| \right) - \kappa_0 \leq 0  \ \  {\normalfont{\text{ on }}} \partial V 
$$
for a given $\kappa_0$, where $\partial^{\pm}_i u(x) = u(x \pm e_i) - u(x)$. In light of this observation one may consider  a wider class of problems in terms of general 
prescribed boundary mass given by $F(x, \partial_1^{\pm} u , ... , \partial_d^{\pm}  u)$, where $F$ defined  on $\Z^d \times \R_+^{2d}$
has to satisfy some properties (ellipticity and other), yet to be found. 
The methods developed in this paper seem to have a good chance to go through for at least ``nice enough" function $F$.
This aspect we shall leave for interested reader to explore.
\end{remark}

\section{Well-posedness of the model}

In this section we prove two basic properties of the boundary sandpile model.
Namely, that the model needs to visit a finite number of sites in $\Z^d$ to reach a stable state,
and that the final stable configuration is independent of the toppling sequence.
For the proofs we will use some ideas from \cite[Lemma 3.1]{Lev-Per} and 
\cite[Lemma 3.1]{Lev-Per10}.

\begin{prop}\label{Prop-stabil-toppling}\normalfont{(Existence of odometer function)}
For a given $\mathrm{BS}(\mu_0, \CAP)$ any infinitive toppling sequence
$T$ is stabilizing.
\end{prop}

\begin{proof}
Let $n:=\sum_{x\in \Z^d} \mu_0(x)$ be the total mass of the sandpile, and set
$V_0 = \supp \mu_0$. For each $k\in \N$
let $V_k$ be the set of visited sites after invoking $k$-th toppling in $T$.
Let also $u_k$ and $\mu_k$ be respectively the odometer function and the distribution of mass after 
the $k$-th site in $T$ has toppled.

We first show that for any $k\in \N$ and for each $x\in \partial V_k$ satisfying
$\mu_0(x) = 0$ one has 
\begin{equation}\label{mu-k-from-below}
\mu_k(x) >\frac{1}{2d} \kappa_0. 
\end{equation}
Indeed, let $1\leq i \leq k$ be the smallest integer such that $x\in V_i \setminus V_{i-1}$,
and let $y\in V_{i-1}$ be the neighbour of $x$ that topples at time $i$.
Since $y\in V_{i-1}$ and $y\sim x \notin V_{i-1}$ we have $y\in \partial V_{i-1}$,
but as toppling of $y$ is legal (since it is producing a previously unvisited site $x$), we must have $\mu_{i-1}(y) > \kappa_0$.
Consequently we get $\mu_i(x) >\frac{1}{2d} \kappa_0$,
and since $x$ did not topple at steps $i+1,...,k$, as otherwise $x$
will be in the interior of $V_k$, we get $\mu_k(x)=\mu_i(x)$
and hence (\ref{mu-k-from-below}).
Comparing (\ref{mu-k-from-below}) with the total mass $n$, for any $k\geq 0$ we get 
\begin{equation}\label{bound-on-bdry-points}
| \partial V_k \setminus V_0 | \leq 2d \frac{n}{\kappa_0},
\end{equation}
as an upper bound for the number of newly emerged\footnote{Here one can only estimate the
 number of boundary points of the visited sites which were not initially in the support of $\mu_0$.
 The simple reason is that in general $\mu_0$ may have, say, arbitrarily large number of isolated points with small mass,
 which may never be visited again in the course of the process.} boundary points
of visited sites at any time of the process. 

Now assume that $V_k$ is connected as a lattice graph. Fix $1\leq \alpha \leq d$
and let $a, b \in \partial V_k$ be such that
$a_\alpha  = \min\limits_{x\in \partial V_k} x_\alpha $ and
$b_\alpha  = \max\limits_{x\in \partial V_k} x_\alpha $. 
For each integer $a_\alpha \leq c \leq b_\alpha$ consider the hyperplane $H_c=\{x\in\R^d: \ x_\alpha  = c \}$.
It is clear that $H_c \cap V_k$ contains at least one element of $\partial V_k$
and these points are pairwise different for integer $c\in [a_\alpha, b_\alpha]$.
By (\ref{bound-on-bdry-points}) one has
$$
b_\alpha - a_\alpha \leq 2d \frac{n}{\kappa_0} + |V_0|.
$$
As $1\leq \alpha \leq d$ is arbitrary, for any $x\in V_k$
we obtain 
\begin{equation}\label{vis-sites-bound}
\max\limits_{1\leq \alpha \leq d } \max\limits_{y\in V_0} |(x-y)_{\alpha}|  \leq 2d \frac{n}{\kappa_0} + |V_0|,
\end{equation}
hence the process occupies a finite region of $\Z^d$
independently of\footnote{Later on in Lemma \ref{Lem-balls-inside-out} we will obtain
more precise bounds for the set of visited sites. However, the bound in Lemma \ref{Lem-balls-inside-out}
uses a priori boundedness of the set of visited sites.} $k$.
It is clear that assuming connectedness of $V_k$ resulted in no loss of generality,
as we can employ the scanning procedure by hyperplanes on each connected component
of the lattice graph of $V_k$.

Since the toppling is infinitive, for each $k\geq 1$ all unstable sites on $\partial V_k$
will topple eventually, by so extending the set of visited sites.
But (\ref{vis-sites-bound}) shows that after finite number of topplings,
the boundary of visited sites will become stable and will remain so throughout.
Precisely, there is $k\geq 1$ and a set $V\subset \Z^d$
such that $V_i  = V$ for all $i\geq k$.
By (\ref{Laplace-of-odometer}) for any integer $k\geq 0$  we have
\begin{equation}\label{u-k-Laplacian}
\Delta u_k = \mu_k - \mu_0 \text{ in } \Z^d.
\end{equation}
The \emph{r.h.s.} of the last expression, as well as each $u_k$ are supported in $V$
for all integer $k\geq 0$. We thus have that the functions $x \mapsto \Delta u_k(x)$, independently of $k\in \N$,
are uniformly bounded on $\Z^d$, and have compact supports.
Take any $x\in \partial V$, then $u_k(x)=0$ and $\mu_k(x) \leq \CAP$ in view of the stability of $V$. Hence
$$
\frac{1}{2d} \sum\limits_{y\sim x} u_k(y) = \Delta u_k(x) =\mu_k(x) - \mu_0(x) \leq \CAP,
$$
which shows that $\{u_k(y)\}_{k=1}^\infty$ is uniformly bounded for any $y\sim x$.
This uniform bound on $u_k$ propagates into each connected component of the lattice graph of $V$.
Hence we get $u_k(x) \leq C$ for any $x\in \Z^d$ and any $k\geq 1$ with some constant $C=C(d,n)$.

Since for each $x\in \Z^d$, the sequence $\{u_k(x)\}_{k=1}^\infty$ is increasing and bounded above, it has a finite limit
which we denote by $u(x)$. Clearly $u$ is non-negative and compactly supported.
It also follows that the \emph{l.h.s.} of (\ref{u-k-Laplacian}) must have a limit everywhere in $\Z^d$.
Passing to the limit in (\ref{u-k-Laplacian}) as $k$ goes to infinity, 
we get $\Delta u = \mu - \mu_0 $ in $\Z^d$ where by $\mu$ we denote the limit distribution of $\mu_k$.

To complete the proof it is left to observe that the final distribution $\mu$
is stable. As we saw above, $\mu$ is stable on $\partial V$, hence it remains to show that
$\mu=0$ in the interior of $V$. Fix any $x\in \INT{V}$, since $T$
is infinitive we get $\mu_k(x) =0$ for infinitely many $k$. But as $\mu(x)  = \lim\limits_{k\to \infty}\mu_k(x)$
and the limit exists, we obtain $\mu(x) =0$ completing the proof of the proposition.
\end{proof}

\begin{remark}
It is clear from the proof of Proposition \ref{Prop-stabil-toppling}
that one can require the toppling to be infinitive only on the set of visited sites.
We do not do so as the final domain is not known a priori.
\end{remark}

We call the function $u$ constructed in the proof of Proposition \ref{Prop-stabil-toppling}
the \emph{odometer} and the distribution $\mu$ the \emph{final configuration}
corresponding to the given toppling sequence.

\begin{prop}{\normalfont{(Abelian property)}}\label{Prop-Abelian}
For a given $\mathrm{BS}(\mu_0, \CAP)$ and any two infinitive toppling sequences $T_1$ and $T_2$, one has
$u_1=u_2$ for the corresponding odometer functions.
\end{prop}

\begin{proof}
As in Proposition \ref{Prop-stabil-toppling} by $n$ denote the total mass of the system.
In view of Proposition \ref{Prop-stabil-toppling} each $T_i$
is stabilizing and has well-defined odometer function, call it $u_i$.
We will prove that $u_2(x) \geq u_1(x)$ for any
$x\in \Z^d$, which by symmetry will imply the desired result.

For $i=1,2$ by $u_{i,k}$ denote the odometer function and by $\mu_{i,k}$ the distribution 
corresponding to the sequence $T_i$ after the $k$-th toppling occurs.
Let $T_1= \{x^{(k)}\}_{k=1}^\infty$,
we now show that 
\begin{equation}\label{ineq-abelian}
u_2(x^{(k)}) \geq u_{1,k} (x^{(k)}), \qquad k=1,2,... \ . 
\end{equation}
Observe that (\ref{ineq-abelian}) is enough for our purpose
since the site $x^{(k)}$ appears infinitely often in $T_1$ and $u_{1,k} $ converges to $u_1$ pointwise as
$k$ tends to infinity. 
Thus, in what follows we prove (\ref{ineq-abelian}) which we will do by induction on $k$.

When $k=1$, (\ref{ineq-abelian}) is trivially true. Now assume that it holds for any time $1\leq i<k$.
We divide the analysis into two cases.

\vspace{0.2cm}

\noindent \textbf{Case 1.} \emph{Toppling $x^{(k)}$ in $T_1$ at time $k$ is not a legal move}.
\vspace{0.1cm}

In this case we have
\begin{equation}\label{non-legal}
u_{1,k}(x^{(k)}) = u_{1,k-1}(x^{(k)}).
\end{equation}
So, if $x^{(k)}$ has never toppled in $T_1$ prior to time $k$, then $u_{1,k-1}(x^{(k)})=0$
and we are done. Otherwise, let $i\leq k-1$ be the last time $x^{(k)}$
has toppled in $T_1$. We thus have $x^{(i)} =x^{(k)} $,
hence, using the inductive hypothesis, we conclude
$$
u_2(x^{(k)} ) =u_2(x^{(i)}) \geq u_{1,i}(x^{(i)}) = u_{1,k-1}(x^{(i)}) =  u_{1,k}(x^{(k)}),
$$
where the last equality follows from (\ref{non-legal}).
This completes the induction step for the case of a non-legal move.

\vspace{0.2cm}

\noindent \textbf{Case 2.} \emph{Toppling $x^{(k)}$ in $T_1$ at time $k$ is legal}.
\vspace{0.1cm}

We first show that for any $x\neq x^{(k)}$  one has
\begin{equation}\label{u2-geq-u1k}
u_2(x)\geq u_{1,k} (x).
\end{equation}
There are two possible sub-cases, either $x$ was never toppled in $T_1$ up to time $k$,
which implies $u_{1,k} (x)=0$ and we get (\ref{u2-geq-u1k}),
or  $x$ was toppled at some time before $k$.
In the latter case let $i\leq k$ be the last time $x$ has toppled prior to time $k$.
Observe that $i<k$ since $x\neq x^{(k)}$. Then we have
$$
u_{1,k} (x) = u_{1,i}(x) \leq u_2(x),
$$
where the second inequality follows by inductive hypothesis.
We thus have proved (\ref{u2-geq-u1k}) for all $x\neq x^{(k)}$.
Consider the following inequality
\begin{equation}\label{mu2-dist-vs-mu1k}
\mu_2(x^{(k)}) \leq \mu_{1,k} (x^{(k)}).
\end{equation}
To complete the induction step suppose for a moment that (\ref{mu2-dist-vs-mu1k}) holds
true. Then
\begin{multline*}
\frac{1}{2d} \sum\limits_{y \sim x^{(k)} } u_2(y)  - u_2(x^{(k)} ) = 
  \Delta u_2( x^{(k)} ) =\mu_2(x^{(k)}) - \mu_0(x^{(k)}) \stackrel{ (\ref{mu2-dist-vs-mu1k}) }{\leq} \\ 
  \mu_{1,k} (x^{(k)}) - \mu_0(x^{(k)}) = 
  \Delta u_{1,k}(x^{(k)}) =\frac{1}{2d} \sum\limits_{y \sim x^{(k)} } u_{1,k}(y)  - u_{1,k}( x^{(k)} ).
\end{multline*}
Rearranging the first and the last terms leads to
$$
u_2( x^{(k)} ) - u_{1,k}( x^{(k)} ) \geq \frac{1}{2d}\sum\limits_{y \sim x^{(k)} } \left( u_2(y) - u_{1,k} (y) \right) \geq 0,
$$
where the last inequality is due to (\ref{u2-geq-u1k}). This completes the induction.
Thereby, to finish the proof of the proposition, we need to verify (\ref{mu2-dist-vs-mu1k})
which we do next. 

To prove (\ref{mu2-dist-vs-mu1k}) we will assume that $x^{(k)}$ is on the boundary of the visited sites generated by $T_2$,
call it $V_2 \subset \Z^d$,
as otherwise the \emph{l.h.s.} of (\ref{mu2-dist-vs-mu1k}) vanishes and the inequality is trivial.
Recall, that we are in Case 2, hence there are two possible reasons for $x^{(k)}$ to topple in $T_1$ at time $k$. Namely,
\begin{itemize}
 \item[(a)] $x^{(k)}$ was on the boundary of visited sites of $T_1$ at time $k-1$ and had mass strictly greater than $\CAP$,
 \item[(b)] $x^{(k)}$ was in the interior of visited sites of $T_1$ at time $k-1$ and had positive mass.
\end{itemize}
We show that either of the cases leads to a contradiction.
Assume (a), then
\begin{align*}
 \CAP < \mu_{1,k-1}( x^{(k)} ) &= \Delta u_{1, k-1} (x^{(k)}) - \mu_0 ( x^{(k)} )   \\
  &= \frac{1}{2d} \sum\limits_{y\sim x^{(k)}} u_{1,k-1}(y) - \mu_0 (x^{(k)})  \ \ \  ( \text{since  } u_{1,k-1} ( x^{(k)} ) =0 )  \\
  &\leq \frac{1}{2d} \sum\limits_{y\sim x^{(k)}} u_2 (y) - \mu_0 (x^{(k)})  \ \ \ \ \ \ \  \left( \text{from } (\ref{u2-geq-u1k}) \right) \\
  &= \Delta u_2(x^{(k)}) - \mu_0(x^{(k)}) \ \ \ \qquad \  \ \ \ \  ( \text{since }  u_2(x^{(k)})=0 ) \\
  &= \mu_2( x^{(k)} ) - \mu_0(x^{ (k) }) \leq \CAP \ \ \ \ \ \ \ \ \ \left( \text{as } T_2 \text{ is stabilizing}  \right),
\end{align*}
which is a contradiction.

Now consider the case when $x^{(k)}$ topples according to (b).
Recall that $x^{(k)} \in \partial V_2$.
This implies that $x^{(k)}$ cannot topple in $T_1$ prior to time $k$.
Indeed, if $x^{(k)}$ topples for some $i<k$, then $u_{1,i}(x^{(k)})>0$, but by inductive hypothesis we have
$0=u_2(x^{(k)})\geq u_{1,i} (x^{(k)}) >0$ which is false.
Next let $i<k$ be the first time when the toppling of $T_1$ produces a vertex outside $V_2$,
and let $y$ be the site that topples at time $i$. On one hand $y$ has to be on the boundary of the set visited by $T_2$,
on the other hand, as we saw already, $y$ must be different from $x^{(k)}$.
Hence by (\ref{u2-geq-u1k}) and the fact that $y\in \partial V_2$ we get $u_{1,i} (y)\leq u_{1,k}(y)=0$ meaning that
$y$ cannot topple. We arrive at a contradiction,
hence (\ref{mu2-dist-vs-mu1k}) is proved, and so is the proposition.
\end{proof}

A simple consequence of the abelian property is the following
lattice symmetries of the sandpile.
Consider the set of directions
\begin{equation}\label{dir-N}
\mathcal{N}=\{e_i, e_i+e_j, e_i - e_j, \ 1\leq i\neq j \leq d\} \subset \Z^d,
\end{equation}
and for a given $e\in \mathcal{N}$ let $T_e$ be the 
hyperplane through the origin with a normal vector collinear with $e$.
Then, $T_e$ is a mirror symmetry hyperplane of the unit cube $[-1/2, 1/2]^d \subset \R^d$,
and any symmetry hyperplane of this cube is of the form $T_e$ for some $e\in \mathcal{N}$.
Mirror reflections with respect to a hyperplane $T_e$ preserve the lattice $\Z^d$.
These symmetries are inherited by sandpile with single source as we see next.

\begin{cor}\label{cor-O-symm}{\normalfont{(lattice symmetries of the model)}}
Let the initial distribution $\mu_0$ be a point mass supported at the origin,
and let $V$ be its final domain of visited sites, and $u$ be the odometer function. 
Then both $V$ and $u$ are symmetric with respect to any hyperplane $T_e$
where $e\in \mathcal{N}$.
\end{cor}

\begin{proof}
Fix $e\in \mathcal{N}$ and for each $x\in \Z^d$ denote $[x] = \{ x, x_e \}$,
where $x_e \in \mathbb{Z}^d$ is the mirror reflection of $x$ with respect
to the hyperplane $T_e$.
Take any infinitive toppling $T=\{x^{(k)}\}_{k=1}^\infty$ for $\mu_0$ and 
consider its ``symmetrized" version $ [T] =\{ [x^{(k)}]  \}_{k=1}^\infty$,
where each $x^{(k)}$ in $T$ is replaced by its corresponding 2-element class $[x]$.
Clearly the new toppling sequence $[T]$ is infinitive, and hence stabilizing in view of Proposition \ref{Prop-stabil-toppling}.
Notice that since $x$ and $x_e$ are not lattice neighbours, then the topplings of $x$ 
and $x_e$ do not interfere, and can be toppled in any order.
Also, if $u_k$ is the odometer function after toppling both elements of $[x^{(k)}]$ in $[T]$,
then it is symmetric with respect to hyperplane $T_e$ by construction, namely
$u_k(x)= u_k(x_e) $ for any $x\in \Z^d$. Since the odometer $u$ corresponding to $[T]$ is the pointwise limit
of $u_k$, we get the symmetry of the odometer and the set of visited sites with respect to $T_e$.

As $e\in \mathcal{N}$ was arbitrary, the abelian property of the sandpile completes the proof
of the corollary.
\end{proof}

\begin{remark}
One can see, in a similar way, that for any initial distribution $\mu_0$, its final configuration $V$ and odometer function $u$
inherit lattice symmetries of $\mu_0$. We do not formulate this rigorously, as later on we will be using the symmetry for point masses only.
\end{remark}

\section{Toppling procedure on graphs}

It is apparent that our boundary sandpile process requires infinite number of toppling steps
to reach a stable configuration (see Remark \ref{Rem-infinite-steps} below).
Nonetheless, it is not hard to observe that the process can be stabilized effectively 
in finite number of steps by absorbing part of the topplings into certain linear operators
acting on graphs. This fact might also be useful for numerical simulations.

Let $G=(V,E)$ be a finite connected, simple (i.e. without loops and multiple edges) graph.
Partition $V$ into two non-empty subsets $V_0$ and $V_1$ which we call \textbf{source} and \textbf{sink} respectively.
For a given function 
$\mu:V_0 \to \R_+$, in analogy with (\ref{mu-k+1-def}), consider the toppling transformation $T$ acting on a vertex $v\in V_0$ by
\begin{equation}\label{toppl-graph}
(T_v \mu)(u) = \begin{cases}  0, &\text{if $u=v$}, \\ 
 \mu(u) + \frac{\mu(v)}{\mathrm{deg}(v)} ,&\text{if  $ u\sim v    $}, \\
 \mu(u), &\text{otherwise},    \end{cases}
\end{equation}
where $\mathrm{deg}(v)$ is the degree of the vertex $v$ in $G$, which is always positive in view of the connectedness of the graph.
We also denote $\mathbb{E} [\mu] = \frac{1}{|V_0|}\sum\limits_{v\in V_0 } \mu(v) $ for the average
of $\mu$.
With these preparations we easily obtain the following.

\begin{lem}\label{Lem-Graph-balayage}
For any map $\sigma:\N \to V_0$ such that the pre-image of each $v\in V_0$ is infinite,
for all $v\in V_0$ we have 
$$
\lim\limits_{N\to \infty} (T_{\sigma(N)}\circ ... \circ T_{\sigma(1)} \mu)(v) = 0.
$$
\end{lem}

\begin{proof}
We will assume without loss of generality that the subgraph on $V_0$ is connected, as otherwise we can apply
the argument that follows to each connected component of the subgraph.
Set $d_\ast:= \max_{v \in V_0} \mathrm{deg}(v)$ and let $\rho$ be the diameter of the subgraph on $V_0$,
i.e. the maximal length of the shortest path between any two vertices.
Assume that the maximum of $\mu$
is attained at some $v\in V_0$, and take any shortest path from $v$ to the sink $V_1$ through the vertices of $V_0$. Let
$v=v_{1} \sim .... \sim v_{k} \sim v^\ast \in V_1$ be such path where each $v_i \in V_0$.
We have $k\leq \rho$ by definition.
From the definition of $\sigma$ we can choose $N \in \N$ large enough such that 
the sequence $(v_i)_{i=1}^k$ is a subsequence of $(\sigma(i))_{i=1}^N$.
But after toppling $v_1,...,v_k$ the total mass of $V_0$ is decreased
by at least $\mu(v_1)/ d_\ast^\rho$, and as $\mu(v_1)=\max_{v\in V_0} \mu(v)$, we have
\begin{equation}\label{exp-decay}
\mathbb{E}[ T_{\sigma(N)}\circ ... \circ T_{\sigma(1)} \mu] \leq \mathbb{E}[\mu] 
\left(1- \frac{1}{|V_0|} \frac{1}{d_\ast^\rho}  \right).
\end{equation}
Since the average is reduced by a constant multiple the claim follows.
\end{proof}

\begin{remark}\label{Rem-abelian-graph}
By adapting the proof of Proposition \ref{Prop-Abelian} (in fact in a simpler form) one can easily show that the amount of mass each vertex of the sink receives
is independent of the toppling order, as long as each source vertex is toppled infinite number of times.
\end{remark}

\begin{remark}\label{Rem-infinite-steps}
Take a graph with two vertices and a single edge connecting them.
Let also each of these two vertices be attached to any number of sinks (possibly none, but not simultaneously). 
Then clearly one has to topple each vertex infinitely often in order to entirely transfer any given mass to the sinks.
Hence the requirement on each vertex to topple infinite number of times cannot be eliminated.
\end{remark}

We will see shortly that one can avoid a straightforward application of topplings
to transform mass from \emph{source} to \emph{sink}.
By realizing this transference
as a linear operator we will perform the mass transportation from source to sink effectively in one step.

Keeping the notation of Lemma \ref{Lem-Graph-balayage} let the source
be $V_0=\{v_1,...,v_N\}$ and the mass vector be equal to $\mu_0= (m_1,...,m_N)$.
Consider the toppling 
$ T = T_{v_N} \circ ... \circ T_{v_1} $ and for $k\in \N$ let $T^k$
denote the $k$-th iterate of $T$. Then by Lemma \ref{Lem-Graph-balayage}
we have that $T^k \mu_0 \to 0$ as $k \to \infty$ and the goal is to determine 
how much mass each sink vertex will receive in the limit.
Due to the abelian property (Remark \ref{Rem-abelian-graph}) the amount of mass
distributed to the sinks is independent of the toppling sequence,
hence specifying the toppling order results in no loss of generality.
To compute the final distribution of the mass on the sink we introduce matrices $M, D  \in M_N(\R) $ where we call $M$ the \textbf{mass transformation} matrix,
and $D$ the \textbf{mass distribution} matrix. 
More precisely $M_{ij}$ shows the proportion of $\mu_0(v_j)$ which is present at vertex $v_i$ after we have applied the transformation $T$ to the graph.
Namely if $T \mu_0$ is the new mass-vector after applying $T$, then 
$$
[T \mu_0] (v_i) = M_{i1} \mu_0(v_1) +M_{i2} \mu_0(v_2)+...+M_{i N} \mu_0(v_N).
$$
In a similar fashion $D_{ij}$ shows the proportion of $\mu_0(v_j)$ transferred to its neighbours by $v_i$.
Thus each sink (if any) attached to $v_i$ gets mass equal to
$$
D_{i1} \mu_0(v_1) + D_{i 2} \mu_0(v_2) + ... + D_{i N} \mu_0(v_N).
$$
A simple bookkeeping procedure illustrated below in Algorithm \ref{Algo-1} allows us to compute these coefficients 
when we consequently topple $v_1 \to v_2 \to .... \to v_N $.

\begin{algorithm}[!htbp]
   \caption{Computation of the transformation matrices}
\begin{flushleft}
 \textbf{Input:} Connected graph $G=(V,E)$ on source vertices $V=\{v_1,...,v_N\}$
\end{flushleft}
   \begin{algorithmic}[1]
   \State Initialize $M$ to $N \times N$ \emph{identity} and $D$ to $N \times N$ \emph{zero} matrices.
        \For{$i = 1$ to $N$} \Comment{corresponds to the toppling of $v_i$}
            \For {$j = 1$ to $N$}
               \State $D_{ij} \leftarrow  D_{ij} + \frac{M_{ij}}{\mathrm{deg}(v_i)}$
            \EndFor         
        
             \For {$p = 1$ to $N$}
                \For {$j $ with $v_j \sim v_i$}
		  \State $M_{jp}  \leftarrow  M_{jp} + \frac{M_{ip}}{\mathrm{deg}(v_i)}$    
                \EndFor
                 \State $M_{ip} \leftarrow 0$ 
             \EndFor   
        \EndFor     
\end{algorithmic}
\begin{flushleft}
 \textbf{Output:} Matrices $M$ and $D$
\end{flushleft}
\label{Algo-1}
\end{algorithm}

Now observe that after the toppling $T^k$, $k\in \N$, is applied,
the mass vector distributed to sinks becomes
\begin{equation}\label{Neumann-series}
\mu_k = D \mu_0 + D [M \mu_0]+...+ D[M^{k-1} \mu_0] = D \left[ \sum\limits_{i=0}^{k-1} M^i \right] \mu_0 \in \R^N.
\end{equation}
To sum the series in (\ref{Neumann-series}) we need the following.

\begin{lem}\label{Lem-e-value1}
All eigenvalues of $M$ are less than 1 by absolute value.
\end{lem}
\begin{proof}
Take any vector $\mu_0 \in \R^N$ with all coordinates being non-negative.
By virtue of (\ref{exp-decay}) there exists an integer $k\geq 1$ and a constant $0<c<1$
such that 
$$
\mathbb{E} [ M^k \mu_0 ] \leq (1-c) \mathbb{E} [\mu_0].
$$
Since $M$ has non-negative entries due to construction by Algorithm \ref{Algo-1}, from the last inequality we get
$$
|| M^k \mu_0 ||_{l^1 }  \leq (1-c) ||\mu_0||_{l^1}
$$
now for any $\mu_0\in \R^N$.
Hence, if $\mu_0\in \R^N \setminus \{0\}$ is an eigenvector of $M$
corresponding to an eigenvalue $\lambda$, we get
$$
||M^k \mu_0||_{l^1}= |\lambda|^k ||\mu_0||_{l^1} \leq (1-c) ||\mu_0||_{l^1},
$$
which implies $|\lambda|^k \leq (1-c)<1$ completing the proof of the lemma.
\end{proof}

Lemma \ref{Lem-e-value1} implies that the matrix $\mathrm{Id}- M$ is invertible,
where $\mathrm{Id}$ is the $N\times N$ identity matrix.
Hence the Neumann series of $M$, as we have in (\ref{Neumann-series}),
is convergent. Passing to limit in (\ref{Neumann-series}) as $k \to \infty $
demonstrates that the total mass $\mu$ distributed to the sinks equals
\begin{equation}\label{mass-trans-sink}
 \R^N \ni \mu = D \left[ \sum\limits_{i=0}^\infty M^i \right] \mu_0 = D (\mathrm{Id}- M)^{-1} \mu_0. 
\end{equation}
Observe that all mass is being transferred to the sink by Lemma \ref{Lem-Graph-balayage}, and hence there is no loss of mass. This shows that the linear operator
defined by (\ref{mass-trans-sink}) is non-degenerate.
Hence, the matrix $D$ is invertible as well. 
Finally, each sink vertex attached to $v_i$ gets
mass equal to the $i$-th component of the vector 
\begin{equation}\label{sink-distrib}
 [D(\mathrm{Id}- M)^{-1} \mu_0]\in \R^N.
\end{equation}

Based on the above discussion, we next present a simple pseudocode, which allows
to stabilize the given initial distribution in finite number of steps.
 
\begin{algorithm}[!htbp]
   \caption{Boundary sandpile in finite steps}
\begin{flushleft}
 \textbf{Input:} Initial distribution $\mu_0$ and boundary capacity $\CAP$
\end{flushleft}
 
\begin{algorithmic}[1]
 \State Initialize $V$, the set of visited sites, as the support of $\mu_0$
 \vspace{0.1cm}
 \State Populate a stack $\mathcal{V}$ of sites from $V$ that need to be toppled.
 
\Repeat
     \While{$\mathcal{V}$ is NOT empty}
		\State topple(v) \Comment $v$ is the top element of the stack
		\For {$w\in \Z^d$ with $w\sim v$ }
		 \If{ mass at $w$ is strictly larger than $\CAP$}
		   \State include $w$ in $\mathcal{V}$ if not there already
		 \EndIf
		\EndFor
        \State remove $v$ from $\mathcal{V}$
        
        \EndWhile

        \For {each connected component $V_k$ of the visited sites $V$}
          \State construct $G$, the lattice graph on vertices of $\INT{V_k}$
          \State construct the distribution matrices $M$ and $D$ for $G$ by Algorithm \ref{Algo-1}
          \State use $M$ and $D$ to distribute the mass from $\INT{V_k}$ to $\partial V_k$ by \eqref{sink-distrib}
        \EndFor
        
    \State Re-populate the stack $\mathcal{V}$ by elements of $\partial V$ that need to be toppled
    \vspace{0.1cm}
 \Until{$\mathcal{V}$ is empty}       
\end{algorithmic}
\begin{flushleft}
 \textbf{Output:} Final set of visited sites and stable mass distribution on the boundary
\end{flushleft}

\label{Algo-2}
\end{algorithm}

In Algorithm \ref{Algo-2} we simply move the boundary of the set of visited
sites as long as there is a possibility. Then, we transfer at once all mass
from the interior onto the boundary which requires a passage to the limit
in a toppling procedure. Consequently,
the fact that this will produce the same final configuration
as the original boundary sandpile process, is not a direct corollary of abelian property
of the sandpile, but follows from Lemma \ref{Lem-interim}.
In particular, due to Lemma \ref{Lem-interim} each iteration of the algorithm
resuming from Step 3, starts with an initial distribution $\mu$ for which
the two sandpiles $\mathrm{BS}(\mu, \CAP)$ and $\mathrm{BS}(\mu_0, \CAP)$
produce the same final configurations. This fact, coupled with Proposition \ref{Prop-stabil-toppling}
implies that Algorithm \ref{Algo-2} terminates in finite number of steps
since the set of visited sites cannot grow to infinity
\footnote{We do not
attempt to estimate the complexity of Algorithm \ref{Algo-2} depending on $\mu_0$ and $\CAP$,
however, it is interesting to observe, that the same approach can stabilize the divisible sandpile of Levine-Peres \cite{Lev-Per}
in finite number of steps.}.

\section{Discrete potential theory in the sandpile framework}

\subsection{The canonical domain of the model}

We start with the following.

\vspace{0.3cm}
\noindent \textbf{Discrete maximum principle (DMP): } Let $u,v: \Z^d \to \R$, $V\subset \Z^d$ be finite, and $\Delta u \geq \Delta v$ in the interior of $V$.
Then $\max_{\partial V} (u-v) \geq \max_{V} (u-v)$.
\vspace{0.1cm}

The proof of this important statement is an easy exercise (see e.g. \cite[Exercise 1.4.7]{Lawler-book-walks}).
We now show that the boundary sandpile requires the smallest possible domain
in the discrete space, where it can stabilize itself.

\begin{definition}\label{Def-stab-pair}{\normalfont{(Stabilizing pair)}}
For an initial distribution $\mu_0$ and a boundary capacity $\CAP$, we say that the pair
$(V, u)$ is stabilizing for $\mathrm{BS}(\mu_0, \CAP)$ if the following holds:
\begin{itemize}
 \item[{\normalfont{(i)}}] $V\subset \Z^d$ is a finite set 
 {\normalfont{and}} $\mathrm{supp} \mu_0 \subset V$,
 \vspace{0.2cm}
 \item[{\normalfont{(ii)}}] $u:V \to \R_+$ is a function satisfying
 $$
  \begin{cases}  \Delta u = -\mu_0 , &\text{{\normalfont{in}} $\INT{V}$}, \\ 
                 u=0   ,&\text{{\normalfont{on}}  $  \partial V$}, \\
                 \mu_0 + \Delta u \leq \CAP   ,&\text{{\normalfont{on}}  $  \partial V$}.
  \end{cases}
$$
\end{itemize}
In this setting, we will call $u$ odometer function.
\end{definition}

\begin{remark}\label{rem-stabilizing}
For any $\mu_0$ and $\CAP$ let $V_0$ be the set of visited sites of the corresponding $\mathrm{BS}$ and $u_0$ be the odometer function.
Then clearly the pair $(V_0,   u_0)$ is stabilizing in a sense of Definition \ref{Def-stab-pair}.
\end{remark}

\begin{theorem}\label{Thm-canonical}{\normalfont{(The minimality principle)}}
Let $V_0\subset \Z^d$ be the set of visited sites of $\mathrm{BS}(\mu_0 , \CAP)$.
Then $V_0$ is the intersection of all $V\subset \Z^d$ 
for which there is a function $u$ such that the pair $(V,u)$ is stabilizing for $\mathrm{BS}(\mu_0, \CAP)$.
\end{theorem}

\begin{proof}
The set of stabilizing pairs is non-empty in view of Remark \ref{rem-stabilizing}.
Hence if $V_*$ is the intersection in the formulation of the theorem, then $V_*$ is finite and contains the support of $\mu_0$.

First, we show that for some $u_*$ the pair $(V_*, u_*)$ is stabilizing.
To see this for $V_*$, it is enough to prove that for any two stabilizing pairs $(V_i, u_i)$, $i=1,2$
the set $V := V_1 \cap V_2$ has the property.
To this end, consider the function $v(x) := \min\{u_1(x), u_2(x) \}$ in $V$.
Clearly
\begin{equation}\label{bdry-as-union}
\partial V = (\partial V_1 \cap V_2 )\cup (\partial V_2 \cap V_1 ),
\end{equation}
from which we get $v=0$ on $\partial V$.
It is also clear that $\Delta v \leq -\mu_0$ in $\INT{V}$.
Next, using the fact that $v$ vanishes on $\partial V$
and is the pointwise minimum of $u_1$ and $u_2$, from (\ref{bdry-as-union})
we have 
\begin{equation}\label{v0-on-V0}
\Delta v +\mu_0 \leq \CAP  \text{ on } \partial V.  
\end{equation}

Now let $w$ be the unique solution to 
$$
\Delta w = -\mu_0 \ \text{ in }   \INT{V} \ \ \text{ and } \ \  w=0 \text{ on } \partial V.
$$
Since $w-v$ is a subsolution in the interior of $V$,
from DMP we have $w \leq v $ in $V$.
Combining this with the fact that both $w$ and $v$ vanish on $\partial V$ we 
arrive at
$$
\Delta w \leq \Delta v   \text{ on } \partial V.
$$
But then, 
$$
\Delta w  +  \mu_0 \leq \Delta v  + \mu_0 \leq \CAP \text{ on } \partial V,
$$
where the last inequality is due to (\ref{v0-on-V0}).
Hence, the pair $(V, w)$ is stabilizing and the claim for $V_*$ follows.

We are now in a position to show $V_0 = V_*$.
By definition of $V_*$ we have $V_*\subset V_0$.
Take any toppling sequence $T=\{ x^{(i)} \}_{i=1}^\infty$ 
where $x^{(i)} \in \INT{V_*}$ for all $i\in \N$ and $T$
is infinitive\footnote{Observe that $\INT{V_*}$ might be empty, in which case $T$ is empty as well, that is no toppling occurs.
This corresponds to the case, when the initial distribution $\mu_0$ is already stable.} on $\INT{V_*}$.
Observe that if $T$ is stabilizing for the original $\mathrm{BS}$, then $V_* = V_0$ as the final configuration is unique
by abelian property of the sandpile. Hence, it is left to prove that $T$ is stabilizing.
Assume it is not. Invoking $T$,
we get that all mass from the interior of $V_*$ is being redistributed onto $\partial V_*$ by Lemma \ref{Lem-Graph-balayage}.
We have  $V_* \subset V_0$, and since $T$ is not stabilizing it follows that the distribution of mass on $V_*$ obtained after applying $T$
is not stable. Since no mass remains in the interior of $V_*$ it follows that the boundary of $V_*$ is unstable.
Let $u$ be the odometer function corresponding to the toppling $T$. Then, by definition we have
\begin{equation}\label{u-in-int-V-star}
\Delta u = - \mu_0 \text{ in } \INT{V_*} \text{ and } u=0 \text{ on } \partial V_*.
\end{equation}
Let also $u_*$ be the odometer for which $(V_*, u_* )$ is stabilizing,
whose existence was established above. By definition, $u_*$ satisfies (\ref{u-in-int-V-star})
hence, due to uniqueness of solutions, we get that $u=u_*$.
But then
$$
\Delta u + \mu_0 = \Delta u_* +\mu_0 \leq \CAP \text{ on } \partial V_*
$$
where the inequality follows since the pair $(V_*, u_*)$ is stabilizing.
This shows, that $(V_*, u)$ is stabilizing as well,
which contradicts our assumption that the toppling $T$ was not stabilizing, and completes the proof of the theorem.
\end{proof}

\begin{definition}\label{Def-stab-pair}{\normalfont{(The stabilizing pair)}}
For a given $\mathrm{BS}(\mu_0, \CAP)$ we call $(V, u)$
\emph{the stabilizing pair} of the sandpile if $V$ is minimal in a sense of Theorem \ref{Thm-canonical}.
\end{definition}

\begin{remark}
An immediate reformulation of the previous theorem is the following 
characterization of the odometer as the smallest supersolution in a certain class of functions on $\Z^d$.
For $\mathrm{BS}(\mu_0, \CAP)$ let $u$ be the odometer function. Then
$$
u= \inf \{ u: \Z^d \to \R_+: \ \Delta u \leq - \mu_0 \text{ in } \INT{\mathcal{U}} {\normalfont\text{  and  }} \mu_0 + \Delta u \leq \CAP \text{ on } \partial \mathcal{U} \},
$$
where $\mathcal{U} : = \{x \in \Z^d: \ u(x)>0 \} \cup \supp \mu_0$. 
\end{remark}

A useful corollary of the minimality principle, which is being used in Algorithm \ref{Algo-2} above, is the following.

\begin{lem}\label{Lem-interim}
Let $(V_0,u_0)$ be the stabilizing pair for $\mathrm{BS}(\mu_0, \CAP)$.
Fix any $V_1 \subset V_0$ such that $\supp \mu_0 \subset V_1$ and let $u_{0\to 1}$ be the solution to
$$
\Delta u_{0\to 1} = -\mu_0 \text{ in }  \INT{V_1} {\normalfont\text{  and  }} u_{0\to 1}=0 \text{ on }  \partial V_1. 
$$
Then the final configurations of $\mathrm{BS}(\mu_0, \CAP)$ and $\mathrm{BS}(\Delta u_{0\to 1}  +\mu_0, \CAP ) $
are identical.
\end{lem}

The lemma states that if we balayage $\mu_0$ to $\partial V_1$ first, and then
balayage the resulting measure from $\partial V_1$ to $\partial V_0$, the boundary measure produced on $\partial V_0$
will be the same as if we had balayaged $\mu_0$ on $\partial V_0$.
This is obvious from a potential theoretic and balayage perspective. 
It is however not straightforward from a $\mathrm{BS}$ dynamic point of view.
\vspace{0.2cm}

\begin{proof}[Proof of Lemma \ref{Lem-interim}]
Set $\mu_1: = \Delta u_{0\to 1}  +\mu_0$, which, by the definition of $u_{0\to 1}$,
is a distribution supported on $\partial V_1$. Let $(V_2, u_2)$ be the stabilizing pair
for $\mathrm{BS}(\mu_1, \CAP)$. The aim is to show that $V_2= V_0$ and 
that the final configurations for $\mathrm{BS}(\mu_0, \CAP)$ and $\mathrm{BS}(\mu_1, \CAP)$ coincide.
The latter is equivalent to equality $\mu_0 + \Delta u_0 = \mu_1 + \Delta u_2$ on $\Z^d$,
which in turn is equivalent to 
\begin{equation}\label{config-equiv}
 \Delta u_0 = \Delta(u_2 + u_{0\to 1})  \text{ on } \Z^d
\end{equation}
which we will prove below.

We have
$$
\Delta u_2 = -\mu_1 \text{ in }  \INT{V_2} \text{  and  } u_2=0 \text{ on }  \partial V_2,
$$
as well as $\mu_2 := \mu_1 + \Delta u_2 \leq \CAP$ on $\partial V_2$. Clearly $V_1 \subset V_2$
therefore $u_{0\to 1} + u_2 =0$ on $\partial V_2$. We also have
$$
\Delta(u_{0\to 1} + u_2 ) = \mu_1 - \mu_0 + \mu_2 -\mu_1 = \mu_2 - \mu_0 \text{ in } V_2,
$$
hence the pair $(V_2, u_{0\to 1} + u_2 )$ is stabilizing for $\mathrm{BS}(\mu_0 , \CAP)$.
As a consequence of Theorem \ref{Thm-canonical} we obtain $V_0 \subset V_2$.
We now prove the reverse inclusion. Since $V_1 \subset V_0$, it follows that $\mu_1$ is supported in $V_0$.
Hence, we may define $u_{1\to 0}$ as the unique solution of
$$
\Delta u_{1\to 0} = -\mu_1 \text{ in }  \INT{V_0} \text{  and  } u_{1\to 0}=0 \text{ on }  \partial V_0.
$$
Since $V_1\subset V_0$ it follows that $ u_{0\to 1} +u_{1 \to 0}=0$ on $\partial V_0$. This implies that the function
$w : = u_0 - (u_{0\to 1} +u_{1 \to 0})$ is vanishing on $\partial V_0$. But we also have
$$
\Delta w = -\mu_0 + \mu_0 - \mu_1 + \mu_1 =0 \text{ in } \INT{V_0},
$$
and by the uniqueness of solutions to the Dirichlet problem, we infer that $w=0$ in $V_0$.
In particular it follows that
$$
\mu_1 + \Delta u_{1 \to 0} = \mu_0 + \Delta u_0,
$$
which, by the stability of $u_0$, implies that $(V_0, u_{1\to 0})$ is stabilizing for $\mathrm{BS}(\mu_1, \CAP)$.
Again, by the minimality principle we get $V_2 \subset V_0$, consequently $V_2= V_0$.
It is left to show (\ref{config-equiv}). Since $V_2= V_0$, the function
$u=u_2 + u_{0\to 1} - u_0$ vanishes on $\partial V_0$ and is identically 0 outside $V_0$.
Hence in the interior of $V_0$ we have
$$
\Delta u = -\mu_1 + \Delta u_{0 \to 1} - \Delta u_0 =
 - \mu_1 +(\mu_1 - \mu_0) + \mu_0 =0,
$$
which shows that $u=0$ on $\Z^d$ and hence (\ref{config-equiv}). The lemma is proved.
\end{proof}

\subsection{Directional monotonicity for point masses}

Yet another important consequence of the minimality principle
of Theorem \ref{Thm-canonical} is a special directional monotonicity of the odometer function
corresponding to an initial distribution concentrated at a single point.
This key property of the boundary sandpile will prove crucial in studying the regularity of the boundary of scaling limit
of the process.

Let $\mathcal{S}$ be the set of mirror symmetry hyperplanes of the unit cube $[0,1]^d$.
More precisely, $\mathcal{S}$ consists of the  hyperplanes
$\{x_i = 1/2\}$, $\{x_i = x_j\}$ and $\{x_i=-x_j\}$ where $i,j=1,...,d$ with $i\neq j$.
In particular, $\mathcal{S}$ has $d^2$ elements in total, and the set of vectors \eqref{dir-N}
represents the set of all directions of unit normals to elements of $\mathcal{S}$.

\begin{theorem}\label{Thm-monotonicity}{\normalfont{(Directional monotonicity)}}
Assume $(V,u)$ is the stabilizing pair of $\mathrm{BS}(n\delta_0, \CAP)$ where $n>0$ and $\CAP>0$.
Fix any hyperplane $T\in \mathcal{S}$.
Then, for any $X_1, X_2 \in \Z^d$, such that $X_1 - X_2$ is a non-zero vector orthogonal to $T$,
we have
$$
u(X_1) \geq u(X_2) \ \ \text{ if } \ \ |X_1| \leq |X_2|.
$$
\end{theorem}

\vspace{0.2cm}
This result formalizes the intuition that moving away from the origin, where the total mass of the sandpile was concentrated
initially, emissions of mass should become smaller. There are however restrictions,
partially due to our proof, on the
way one moves on the lattice. Namely the theorem allows to compare points lying on
a ``lattice-line".
The proof is motivated by A.D. Alexandrov's celebrated \emph{moving plane} method \cite{Alexandrov},
see also \cite{Serrin}.

\vspace{0.2cm}

\begin{proof}[Proof of Theorem \ref{Thm-monotonicity}] 
We will assume $|X_1| \leq |X_2|$ and will show that $u(X_1) \geq u(X_2)$.
This clearly implies the claim of the theorem.

There is a unique translated copy of $T$, call it $T_0$, having equal distance from
$X_1$ and $X_2$. For a given $x\in \R^d$ denote by $x^*$ its mirror reflection with respect to the hyperplane $T_0$.
Due to the choice of $T$ and $X_1$, $X_2$ it follows that each lattice point
is being reflected in this way to a lattice point, and since reflection is an involution,
we have that $(\Z^d)^* = \Z^d$, in particular  $X_1 = (X_2)^*$.

Set $V^* := \{ x^*: \ x\in V \}$ for the reflection of $V$, 
and similarly define the reflected odometer function by $u^*(x) : = u(x^*)$, where $x\in \Z^d$.
From the discussion above, we have that $u^*$ is defined on $\Z^d$ and $V^*\subset \Z^d$.
Let $\mathcal{H}_-$ be the closed halfspace of $\R^d$ determined by $T_0$ and containing the origin,
and let $\mathcal{H}_+ = \R^d \setminus \mathcal{H}_- $.
Consider the set $V_T = V_{-} \cup V_+$ where
$V_{-} :=  \mathcal{H}_- \cap V $ and $V_+ := \mathcal{H}_+ \cap V \cap V^*  $.
Our first goal is to show that $V_T = V$.
By definition $V_T \subset V$, so we need to establish the reverse inclusion.
Consider the function
$$
u_T (x) =  \begin{cases}  u(x) , &\text{if $x \in V_-$}, \\ 
\min\{ u(x), u^* (x) \}   ,&\text{if  $  x \in V_+   $}.  \end{cases}
$$
We claim that
\begin{equation}\label{laplace-u-T}
 \Delta u_T(x) \leq \Delta u(x), \qquad  x\in \INT{V_T}.
\end{equation}
Indeed, fix $x$ in the interior of $V_T$. There are two possibilities.
\vspace{0.1cm}

\noindent \textbf{Case 1.} $x\in V_-$. 

In this case (\ref{laplace-u-T}) follows from
\begin{align*}
\Delta u_T(x) = \frac{1}{2d} \sum\limits_{ y\sim x, \  y\in  V_- } u_T(y) +\frac{1}{2d} \sum\limits_{ y\sim x, \ y\in  V_+ } u_T(y) - u(x) &\leq  \\
\frac{1}{2d} \sum\limits_{ y\sim x, \  y\in  V_- } u(y) +\frac{1}{2d} \sum\limits_{ y\sim x, \ y\in  V_+ } u(y) - u(x) &= \Delta u(x).
\end{align*}

\vspace{0.1cm}

\noindent \textbf{Case 2.} $x\in V_+$. 

We have two sub-cases here, namely either $u_T(x) = u(x)$ or $u_T(x) = u^*(x)$.
In the former case, (\ref{laplace-u-T}) follows as in Case 1 above. 
We will therefore assume that
\begin{equation}\label{u-T-as-min}
 u_T(x) = \min\{u(x), u^*(x) \} = u^*(x).
\end{equation}
Then we have 
\begin{align*}
\numberthis \label{u-T-case2} 2d \Delta u_T(x) =  \sum\limits_{ y\sim x, \  y\in  V_- } u_T(y) + \sum\limits_{ y\sim x, \ y\in  V_+ } u_T(y) - 2d u^*(x) &\leq  \\
 \sum\limits_{ y\sim x, \  y\in  V_- } u(y) + \sum\limits_{ y\sim x, \ y\in  V_+ } u^*(y) - 2d u^*(x) &= \\
\sum\limits_{ y\sim x, \  y\in  V_- } u(y) + \sum\limits_{ y^*\sim x^*, \ y\in  V_+ } u(y^*) - 2d u(x^*) .
\end{align*}
Observe, that if the first sum on the last row of (\ref{u-T-case2}) is empty, then we get $\Delta u_T(x) \leq \Delta u(x^*)$.
Otherwise, there are two possible scenarios treated below.

The first one occurs when the hyperplane $T$ is orthogonal to a coordinate axis, i.e. $T$ is one of $\{x\in \R^d: \ x_i=1/2\}$ for some $1\leq i \leq d$.
By assumption there is $y\in V_-$ such that $y\sim x\in V_+$.
But due to the choice of $T$ this is possible if and only if $y=x^*$ and $x^* \sim x$.
In particular, the above-mentioned sum contains a single element,
namely the one arising from $y=x^*$. Taking into account (\ref{u-T-as-min}), from \eqref{u-T-case2} we obtain $\Delta u_T(x) \leq \Delta u(x^*)$.

Consider the second scenario, when the symmetry hyperplane $T$ is not orthogonal to any of the coordinate axes,
and assume that there exists $y\in V_-$ such that $y\sim x\in V_+$.
Here as well, in view of the choice of $T$, such $y$ must lie on the hyperplane $T_0$,
and hence $y\sim x^*$. Thus the first sum in the last row of \eqref{u-T-case2}
covers all lattice neighbours of $x^*$ lying on $T_0$.
Observe that all lattice neighbours of $x^*$ are in $V_-$, so they are either on $T_0$,
or are obtained as reflections of lattice neighbours of $x$ from $V_+$.
The first class of neighbours is covered by the first sum, and the second class by the second sum of  of the last row of \eqref{u-T-case2}.
Hence, the last row of \eqref{u-T-case2} equals $2d \Delta u(x^*)$.

We conclude that in Case 2
\begin{equation}\label{u-T-x-star}
\Delta u_T(x) \leq \Delta u(x^*).
\end{equation}
Finally, notice that for $x\in V_+$ lying in the interior
of $V_T$ we have that $x^*$ is from the interior of $V_T$.
Hence, $\Delta u(x^*) \in \{0, -n\}$ since the origin is in $V_-$.
We thus get \eqref{laplace-u-T} from \eqref{u-T-x-star}.
Next, we proceed to the final stage of the proof, where we show that $u\leq u_T$ everywhere on $\Z^d$.

It is easy to observe that $u_T =0$ on $\partial V_T$
and $\Delta u_T  \leq \CAP $ on  $\partial V_T$.
Now let $v$ be the solution to $\Delta v = -n \delta_0$ in the interior of $V_T$
and $v=0$ on $\partial V_T$.
By DMP we get $v\leq u_T$ everywhere on $V_T$,
hence $\Delta v \leq \Delta u_T \leq \CAP$ on $\partial V_T$  
as both $v$ and $u_T$ vanish on the boundary of $V_T$.
It follows that the pair $(V_T, v)$ is stabilizing for $\mathrm{BS}(n\delta_0, \CAP)$ (notice that the origin lies in $V_T$),
hence the minimality principle due to Theorem \ref{Thm-canonical} in conjunction with the inclusion $V_T \subset V$
implies $V_T = V$.
In particular we get $u\leq u_T$ everywhere on $V$ since the difference $u_T-u$ defines a supersolution in $V$
with zero boundary data.
Lastly, since both $u$ and $u_T$ vanish outside $V$, we get that $u\leq u_T$ on $\Z^d$.
From here, the desired inequality follows easily.
Indeed, since $|X_1|\leq |X_2|$ and the origin lies in $V_-$,
we must have $X_2\in V_+$. Due to construction we also have $X_1 = (X_2)^*$, 
hence 
$$
u(X_2) \leq  u_T(X_2) = \min\{ u(X_1), u(X_2) \} \leq u(X_1).
$$
The proof of the theorem is completed.
\end{proof}

\subsection{Applications to classical Abelian sandpile}\label{subsec-ASM}

The main aim of this subsection is the analysis of the well-known \emph{Abelian sandpile model} (ASM) 
via tools developed in this paper. In particular,
we will show that the scaling limit of ASM generated from an initial configuration of chips placed
at a single vertex has Lipschitz boundary.
While some aspects of the local (inner) geometry of ASM were analysed in \cite{Lev-Peg-Smart},
we believe that this is the first occurrence addressing the geometry of the boundary.

We briefly recall the definition of the process.
ASM is a lattice growth model for configurations of chips
distributed on vertices of $\Z^d$. A vertex carrying at least $2d$ chips topples
giving a single chip to all its $2d$ lattice neighbours, and losing $2d$ chips itself.
If there are no sites with more than $2d-1$ chips, the process terminates.
For any finite non-negative initial configuration of chips, subsequently toppling
all sites with at least $2d$ chips, the process terminates in finite steps.
This process is abelian in the sense that the final configuration of chips is independent of the order of topplings.

Fix $n\in \N$, and let $u$ be the odometer function for the starting configuration of chips $n\delta_0$.
That is, for each $x\in \Z^d$ the value $u(x)$ is the number of topplings required to reach stable configuration in ASM.
Due to abelian property the function $u_n$ is well-defined.
Let also $s :\Z^d \to \Z_+$ be the final (stable) configuration of chips.
Similarly to (\ref{Laplace-of-odometer}) one can easily see that
$$
s(x) = n\delta_0 + \sum\limits_{y\sim x} \left( u(y) - u(x) \right) = 
n\delta_0 + 2d \Delta u(x), \qquad x\in \Z^d.
$$
We will need the following characterization proved by Fey, Levine, and Peres \cite{FLP}.

\vspace{0.2cm}

\noindent \textbf{Least Action Principle}. Let $n$ and $u$ be as above.
Then
$$
u = \min \{ w: \Z^d \to \Z_+: \ n\delta_0 + 2d \Delta w \leq 2d-1 \}.
$$

\vspace{0.2cm}

We have the analogue of Theorem \ref{Thm-monotonicity} for ASM.
For the definition of $\mathcal{S}$ see the paragraph above Theorem \ref{Thm-monotonicity}.

\begin{theorem}\label{Thm-monotonicity-ASM}{\normalfont{(Directional monotonicity for ASM)}}
For a given $n\in \N$ let $u$ be the odometer for initial distribution $n\delta_0$.
Fix any hyperplane $T\in \mathcal{S}$.
Then, for any $X_1, X_2 \in \Z^d$, such that $X_1 - X_2$ is a non-zero vector orthogonal to $T$,
we have
$$
u(X_1) \geq u(X_2) \ \ \text{ if } \ \ |X_1| \leq |X_2|.
$$
\end{theorem}

\begin{proof}[Proof of Theorem \ref{Thm-monotonicity-ASM}] 
We will only sketch the proof as it is a simplification of the proof of Theorem \ref{Thm-monotonicity}.
Assuming $|X_1| \leq |X_2|$ will show that $u(X_1) \geq u(X_2)$,
which implies the claim of the theorem.

Let $T_0$ be as in the proof of Theorem \ref{Thm-monotonicity}, and
for a given $x\in \R^d$ denote by $x^*$ its mirror reflection with respect to the hyperplane $T_0$.
Recall that we have $X_1 = (X_2)^*$.
Define also the reflected odometer function by $u^*(x) : = u(x^*)$, where $x\in \Z^d$.
We have that $u^*$ is defined on $\Z^d$ with values in $\Z_+$.
Let $\mathcal{H}_-$ be the closed halfspace of $\R^d$ determined by $T_0$ and containing the origin,
and let $\mathcal{H}_+ = \R^d \setminus \mathcal{H}_- $.
Set $V_{-} :=  \mathcal{H}_- \cap \Z^d $, $V_+ := \mathcal{H}_+ \cap \Z^d  $,
and define the function
$$
u_T (x) =  \begin{cases}  u(x) , &\text{if $x \in V_-$}, \\ 
\min\{ u(x), u^* (x) \}   ,&\text{if  $  x \in V_+   $}.  \end{cases}
$$
We aim at showing that
\begin{equation}\label{laplace-u-T-ASM}
n \delta_0+ 2d \Delta u_T(x) \leq 2d-1, \qquad  x\in \Z^d.
\end{equation}
First observe, that this inequality clearly implies the result of the theorem.
Indeed, from (\ref{laplace-u-T-ASM}) and the least action principle formulated above 
we obtain $u \leq u_T$ on $\Z^d$.
Next, since $|X_1|\leq |X_2|$ and the origin lies in $V_-$,
we must have $X_2\in V_+$. Due to construction we also have $X_1 = (X_2)^*$, 
hence 
$$
u(X_2) \leq u_T(X_2)  = \min\{ u(X_1), u(X_2) \} \leq u(X_1),
$$
and the result follows. It is thus left to establish (\ref{laplace-u-T-ASM}), which we do next.
Fix any $x\in \Z^d$. There are two possibilities.
\vspace{0.1cm}

\noindent \textbf{Case 1.} $x\in V_-$. \  Similarly to Case 1 of Theorem \ref{Thm-monotonicity} we get
$$
2d \Delta u_T(x) \leq 2d \Delta u(x) \leq 2d-1 -n \delta_0,
$$
where the last inequality is due to definition of $u$.
\vspace{0.1cm}

\noindent \textbf{Case 2.} $x\in V_+$. \ Exactly, as in Case 2 of Theorem \ref{Thm-monotonicity} we get
$$
2d \Delta u_T(x) \leq 2d \Delta u(x^*) \leq 2d-1.
$$

The proof of the theorem is completed.
\end{proof}

For each integer $n\in \N$ let $u_n$ and $s_n$ be respectively the odometer, and final configuration 
of chips for initial distribution $n\delta_0$. 
Define $\overline{u}_n(x)=n^{-2/d} u_n(n^{1/d} x) $ and $\overline{s}_n (x) = s_n(n^{1/d}x)$
where $x\in n^{-1/d} \Z^d$. It was proved by Pegden and Smart \cite{PS}
that there exists a non-negative and compactly supported function $u_0\in C(\R^d\setminus\{0\})$, $s_0\in L^\infty(\R^d)$ satisfying $0\leq s_0\leq 2d-1$
such that $\overline{u}_n \to u_0$ locally uniformly in $\R^d\setminus\{0\}$,
$\overline{s}_n \to s_0$ weak* in $L^{\infty}(\R^d)$ and $\Delta u_0 = s_0 - \delta_0$
in the sense of distributions, where $\Delta$ is continuous Laplacian and $\delta_0$   is Dirac delta at the origin.  The reader is referred to \cite{PS} for the details. 

Concerning ASM, we prove  the following statement.

\begin{theorem}\label{Thm-ASM}
For  $u_0$  as above set
$u_0(0)=+\infty$ and let the set of directions $\mathcal{N}$ be defined as in \eqref{dir-N}.
The following is true:
\begin{itemize}
\item[{\normalfont(a)}] for any $X_1, X_2 \in \R^d$ with the property that $X_2-X_1$ is non-zero and is collinear to any of the vectors of $\mathcal{N}$,
we have
$$
u_0(X_1) \geq u_0(X_2) \ \ \text{ if } \ \ |X_1| \leq |X_2|.
$$
\item[{\normalfont(b)}] $\partial \{u_0>0\}$ is a locally Lipschitz graph.
\end{itemize}
\end{theorem} 

The reason for defining $u_0(0)=+\infty$ is mainly for the mathematical rigour in the proof.
Apart from monotonicity of the limiting odometer function, this result shows that the
boundary of the scaling limit is Lipschitz.
We postpone the proof of this theorem until subsection \ref{subsec-scaled-od}.

\begin{remark}
This method seems to be promising also for Abelian Sandile model with background heights \cite{FLP}.
\end{remark}

\subsection{The true scale of the model} 
The main aim of this section is to study the growth of the boundary sandpile process
for point masses.

\subsubsection{Heuristics}\label{sub-sec-Heuristics} 
It seems intuitively clear that the value $\CAP$ of boundary capacity should influence the size of the visited
sites. Specifically, the smaller $\CAP$ the greater number of boundary points will be necessary in the stable
configuration, resulting in a larger domain of visited sites.
Our aim is to determine, for initial distributions concentrated at the origin,
the size of the boundary capacity which will leave a chance for existence of a scaling limit.
To proceed, let $(V_n, u_n)$ be the stabilizing pair for $\mathrm{BS}(n\delta_0, \CAP)$, with $n>1$ large.
Let $\sigma_n$ be the mass distribution in the stable
configuration. By definition we have that $\sigma_n$ is concentrated on $\partial V_n$,
and the odometer function $u_n $ satisfies
$$
\Delta u_n(x) = \sigma_n(x) -n \delta_0, \qquad x\in \Z^d.
$$
After scaling the lattice by a small factor $h=h(n)>0$ 
we will get 
$$
\Delta^h \widetilde{u}_n (x) = \sigma_n(h^{-1} x) -n \delta_0, \qquad x\in h \Z^d,
$$
where $\widetilde{u}_n (x) := h^2 u_n(h^{-1} x)$ is the scaled odometer defined on $h\Z^d$.
The scaling limit amount to taking $n \to \infty$. We then want $n\delta_0$ to converge in the sense of distributions
to a continuous Dirac at the origin. Since the volume element at the scale $h>0$ is $h^d dx$
we need $n h^d = 1$ hence the correct order of scaling would be $ h=n^{-1/d}$.
But as $h\to 0$ we want $ h V_n$ to neither shrink to a lower dimensional object, nor escape to infinity,
hence $V_n$ should be of \emph{uniform size}\footnote{By size $h^{-1}$ we mean
that $B(0, c_1 h^{-1})\cap \Z^d \subset V_n \subset B(0, c_2 h^{-1}) $ where $B(0,r)$ is the Euclidean ball
centred at the origin, and having radius $r>0$, and $0<c_1<c_2$ are absolute constants.} $h^{-1}$.
The size of $V_n$ will be determined by $\CAP$.
Indeed, the argument of Lemma \ref{Lem-balls-inside-out} shows that for given $\CAP$, and $n>1$ large, 
$V_n$ would have size $R$ where $\CAP \approx n R^{1-d}$. But for the size $R$ we require $R\approx h = n^{1/d}$,
therefore $\CAP \approx n^{1/d}$ is the only choice for boundary capacity to facilitate this requirement.

In the light of these arguments, in dimension $d\geq 2$ we will set $\CAP=n^{1/d}$ for initial mass $n>1$ at the origin.

\begin{lem}\label{Lem-mono-growth}{\normalfont{(Monotone and unbounded growth)}}
For $n\geq 1$ let $V_n$ be the set of visited sites of $\mathrm{BS}(n\delta_0, n^{1/d})$.
We have the following assertions:
\begin{itemize}
 \item[{\normalfont(a)}]  if $1\leq n <m$ then $V_n\subset V_m$. 
 \vspace{0.2cm}
 \item[{\normalfont(b)}] $\bigcup\limits_{n\geq 1} V_n = \Z^d$. 
\end{itemize}

\end{lem}
\begin{proof}
We start with (a). Let $\partial V_m = \{x^{(i)}\}_{i=1}^k$ for some $k\in \N$. Since the Laplacian is a linear operator (cf. \eqref{sink-distrib}),
there is a collection of weights $\{\alpha_i\}_{i=1}^k$ with all $\alpha_i>0$ such that
for any mass $N>0$ at the origin and $\partial V_m$ as target set, 
each site $x^{(i)}$ receives mass $\alpha_i N$.

By definition we have that $V_m $ is stable, hence $\alpha_i m \leq \CAP(m)=m^{1/d}$
and we get $\alpha_i \leq m^{1/d-1}$ for all $1\leq i\leq k$.
From here it follows 
$$
\alpha_i n \leq m^{1/d -1 } n \leq n^{1/d -1 } n = n^{1/d} = \CAP(n),
$$
which means that $V_m$ is stable for mass $n$ at the origin. But as $V_n$
is the smallest among all stable domains for mass $n$ in view of Theorem \ref{Thm-canonical}, we get $V_n \subset V_m$,
completing the proof of part (a).

We proceed to part (b). Let $n\geq 1$ be fixed, then in view of (a) we have $V_n \subset V_m$
for any $m>n$. Hence from Lemma \ref{Lem-interim} to produce $V_m $ we may topple
mass $m$ from the origin to $\partial V_n$ and then stabilize $\partial V_n$ in $V_m$.
As in part (a) we know that each $x^{(i)} \in \partial V_n$ gets mass $\alpha_i m$.
Therefore, for $m>n$ large enough, we will get
$\alpha_i m >m^{1/d } = \CAP(m)$ for all $1\leq i \leq k$. As a result, all points on $\partial V_n$
have to topple in order to stabilize $V_m$. This proves (b),
and finishes the proof of the lemma.
\end{proof}

\begin{lem}\label{Lem-no-touch}
For $n\geq 1$ let $(V,u)$ be the stabilizing pair of $\mathrm{BS}(n \delta_0, \CAP)$.
If $B$ is a subset of $V$ containing the origin in its interior and $v:B \to \R$ is such that 
 $$
  \begin{cases}  \Delta v  \geq -n \delta_0, &\text{{\normalfont{in}} $\INT{B}$}, \\ 
  v=0  ,&\text{{\normalfont{on}}  $ \partial B    $}, \\
 \Delta v >\CAP, &\text{{\normalfont{on}} $\partial B$},    \end{cases}
 $$
then $\partial B \cap \partial V =\emptyset$.
\end{lem}
\begin{proof}
Assume for contradiction that there exists $x_0\in \partial V \cap \partial B$,
and consider the function $w=u-v$ with $u,v$ as above.
Then clearly $\Delta w \leq 0 $ in $\INT{B}$ and $w \geq 0$ on $\partial B$,
hence by DMP we get $w\geq 0$ in $B$, i.e. $u\geq v$ in $B$.
Since both $u$ and $v$ vanish at $x_0$ we obtain
$$
\CAP< \Delta v(x_0) = \frac{1}{2d} \sum\limits_{y\sim x_0} v(y) \leq \frac{1}{2d} \sum\limits_{y\sim x_0} u( y) = \Delta u(x_0) \leq \CAP,
$$
where the last inequality follows by stability of the sandpile.
This leads to a contradiction and completes the proof.
\end{proof}

For $R>0$ let $G_R$ be the Green's
function with pole at the origin for the discrete ball $Z_R$
defined in (\ref{def-Z-R}), i.e.
\begin{equation}\label{Green-def}
  \Delta G_R = -\delta_0 \text{ in } \INT{Z_R} \qquad \text{and} \qquad G_R =0 \text{ on } \partial Z_R.
\end{equation}

\begin{lem}\label{Lem-Greens-Laplace}
There exist constants $R_0>0$ depending on dimension only,
and $0<c_1<C_1$ depending on $R_0$ and dimension $d$,
such that for any $R>R_0 $ one has
$$ 
c_1 R^{1-d}< \Delta G_R (x) < C_1 R^{1-d} \qquad \forall x\in \partial Z_R.
$$
\end{lem}

\begin{proof}
The proof is a straightforward consequence of the random walk counterpart of the problem.
Let $g(x,y):\Z^d \times \Z^d \to \R  $ be the fundamental solution to $\Delta$ in $\Z^d$,
i.e. $\Delta_x g(x,y) = -\delta_0(y-x)$ for all $x,y \in \Z^d$.
The following asymptotics is well-known \cite{Fuk-Uc}, \cite{Uchiyama}, \cite{Lawler-book-walks}
\begin{equation}\label{g-x-y}
 g(x,y) =   \begin{cases}  -\frac{2}{\pi} \log|x-y|  + \gamma_0 + O(|x-y|^{-2}) , &\text{$d=2$}, \\ 
  \frac{2}{(d-2) \omega_d} |x-y|^{2-d} + O(|x-y|^{-d})  , &\text{$d \geq 3$},    \end{cases}
\end{equation}
where $\gamma_0$ is a constant, and $\omega_d $ is the volume of the unit ball in $\R^d$.
By \cite[Proposition 4.6.2]{Lawler-book} 
for each $x\in \INT{Z_R}$ we have
\begin{equation}\label{G-R}
G_R(x) = g(x,0) -  \mathbb{E}^x [ g(S_{\tau_R}, 0 ) ] ,
\end{equation}
where $\tau_R$ is the first exit time from $\INT{Z_R}$ of a simple random walk $\{S_k\}_{k=0}^\infty$ on $\Z^d$
started at $x\in \INT{Z_R}$, i.e. $\tau_R = \min\{k\geq 1: \ S_k\notin \INT{Z_R}\} $.
Expanding the expectation in (\ref{G-R}) and using the fact that
the simple random walk started inside a finite set of $\Z^d$ will almost surely exit it in finite time, 
we get
$$
G_R (x) =  \sum\limits_{z\in \partial Z_R} \mathbb{P}^x [S_{\tau_R} =z ]  \big( g(x,0 ) - g(z,0) \big).
$$
For $x\in \INT{Z_R}$ satisfying $|x|\geq R/2$, applying (\ref{g-x-y}) to the last formula for $G_R$
we obtain
\begin{equation}\label{g1}
G_R(x) = c_d \widetilde{G}_R(x) + O(R^{-d}),
\end{equation}
where 
\begin{equation}\label{g2}
\widetilde{G}_R(x) = \sum\limits_{z\in \partial Z_R} \mathbb{P}^x [S_{\tau_R} =z ]   \begin{cases}  \log \frac{|z|}{|x|} , &\text{$d=2$}, \\ 
  \frac{1}{|x|^{d-2}} - \frac{1}{|z|^{d-2}}  , &\text{$d \geq 3$},    \end{cases}
\end{equation}
and $c_d $ is a constant depending on $d$. Now things are reduced to some elementary computations.

We start with a lower bound of the lemma. Consider the case of $d=2$.
Fix a point $M_0=(a,b) \in \INT{Z_R}$ such that there exists $M_1\sim M_0$ with $M_1 \in \partial Z_R$.
Let $\max\{|a|,|b| \}=|a|$.
In view of the definition of $Z_R$, the neighbour of $M_0$ having the largest norm among all 4 neighbours
must be on the boundary of $Z_R$. Since we also have $|a|\geq |b|$, it follows that as $M_1$ one can take $(a+1,b)$ for $a>0$
and $(a-1,b)$ for $a<0$.
Assume we have $M_1=(a+1,b)$, i.e. $a>0$.
Then
\begin{multline*}
\log \frac{|M_1|}{|M_0|} = 
\frac 12 \log \frac{|M_1|^2}{|M_0|^2} = \frac 12  \log \left( 1+\frac{2a+1}{a^2 + b^2}  \right) = 
\frac 12  \frac{2a+1}{a^2+b^2 } + O(a^{-2}) \geq \\ 
\frac{1}{a} + O(a^{-2}) \geq \frac{1}{\sqrt{2}} \frac{1}{|M_0| } + O(|M_0|^{-2}) \geq \frac{1}{\sqrt{2}} R^{-1} + O(R^{-2}).
\end{multline*}
Observe that $\mathbb{P}^{M_0} \{ S_{\tau_R} =M_1 \} \geq 1/4$
since the random walk can exit from $\INT{Z_R}$ in one step through $M_1$ if started at $M_0$.
On the other hand, from the definition of $Z_R$ we have $|x|<|z|$ for any $x\in \INT{Z_R}$ and any $z\in \partial Z_R$,
hence all terms in the sum of $\widetilde{G}_R$ are non-negative.
Using this, for $M_0 $ and $M_1$ as above we get
$$
\widetilde{G}_R ( M_0 ) \geq  \frac{1}{4} \log \frac{|M_1|}{|M_0|} + O(R^{-2}) \geq 
\frac {1}{4\sqrt{2}}  R^{-1} + O(R^{-2}).
$$
This implies $G_R( M_0 ) \geq c_2 R^{-1}$ for any $M_0 \in \INT{Z_R}$ having a neighbour on $\partial Z_R$.
Since $G_R$ vanishes on $\partial Z_R$ the bound of the lemma follows.

The lower bound in case of $d\geq 3$ is handled similarly.
The upper bound of the lemma follows directly from (\ref{g1})-(\ref{g2}) and definition
of the ball $Z_R$. 
 The details, which we will omit, are elementary.
The proof of the lemma is complete.
\end{proof}

\begin{lem}\label{Lem-balls-inside-out}
There are dimension dependent constants $0<c_1<C_1$ such that
if $V_n$ is the set of visited sites of $\mathrm{BS}(n\delta_0, n^{1/d})$, with $n>1$ large,
then
 $$
 Z_{c_1 n^{1/d}   } \subset V_n \subset Z_{C_1 n^{1/d}   }.
 $$
\end{lem}

\begin{proof}
By Lemma \ref{Lem-mono-growth} part (b) every site of $\Z^d$ is eventually visited by the sandpile if we keep increasing $n$.
Thus we will assume that $n>1$ is so large that the discrete ball $Z_{2R_0}$ is inside $V_n$,
where $R_0$ is fixed from Lemma \ref{Lem-Greens-Laplace}.
Let us also remark, that if for some $n>1$ we have the inclusion $Z_{2R_0} \subset V_n$,
then for all $n'>n$ we get $Z_{2R_0} \subset V_{n'}$ in view of Lemma \ref{Lem-mono-growth} part (a).

Let $G_R$ be as above (see \eqref{Green-def}). Then, the function $G_{R,n} := n G_R$
satisfies
$$
\Delta G_{R,n} = - n\delta_0 \text{ in } \INT{Z_R} \text{ and } G_{n,R} =0 \text{ on } \partial Z_R.
$$
From Lemma \ref{Lem-Greens-Laplace} we have $\Delta G_{R,n} \geq c_1  n R^{1-d}$ everywhere on $\partial Z_R$.
Observe that if $Z_R\subset V_n$ and $n$ and $R$ are chosen so that $\Delta G_{R,n} > n^{1/d}$ on $\partial Z_R$,
then  by Lemma \ref{Lem-no-touch} we have $\partial Z_R \cap \partial V_n =\emptyset$,
i.e. $Z_R $ stays strictly inside $V_n$.
We now choose a sequence of radii
$R_0<R_1<...<R_k$,
where $R_0$ is as above, and for each $1\leq i \leq k$ we have
$Z_{R_i} \subset Z_{R_{i+1}}$ and $Z_{R_{i+1}} \setminus Z_{R_i} \subset \partial Z_{R_{i+1}}$,
where the latter simply means that we enlarge the balls by 1 lattice step at most.
Given this, as long as we have $c_1 n R_i^{1-d} > n^{1/d}$ we get $Z_{R_i}\subset V_n$.
We can therefore increase the radius up to $c_1^{1/(d-1)} n^{1/d}$
obtaining by so the first inclusion of the current lemma.

The upper bound can be obtained similarly,
by starting with a ball of sufficiently large radius, containing $V_n$,
then shrinking the ball by at most 1 lattice step at a time and using the upper bound of
Lemma \ref{Lem-Greens-Laplace} instead.
The proof of the lemma is complete.  
\end{proof}

Our next result gives a uniform lower bound on the odometer function,
which will be needed in the analysis of the shape of scaling limit.
For the proof we need a discrete version of Harnack's inequality.

\begin{theorem}\label{Thm-Harnack}{\normalfont{(Harnack inequality, see \cite[Theorem 1.7.2]{Lawler-book-walks})}}
For any $\alpha<1$ there exists a constant $C_\alpha$ such that
if $f:Z_R \to [0,\infty)$ is harmonic in $\INT{Z_R}$
then $f(x_1) \leq C_\alpha f(x_2)$ for any $|x_1|, |x_2| \leq \alpha R$.
\end{theorem}

\begin{lem}\label{Lem-bounds-below}{\normalfont{(Bounds from below)}}
For any $r_0>0$ small there exists a constant $c_0>0$ depending on $r_0$ and dimension $d$ such that
for any $n>1$ and each $x^{(0)} \in V_n$ satisfying $\mathrm{dist}(x^{(0)}, \partial V_n) \geq r_0 n^{1/d}$
one has
$$
u_n(x^{(0)}) \geq c_0 n^{2/d},
$$
where $(V_n, u_n)$ is the stabilizing pair of $\mathrm{BS}(n\delta_0, n^{1/d})$.
\end{lem}
\begin{proof}
We first observe that the bound of the lemma holds near the origin,
i.e. there exist dimension dependent constants $\alpha_0>0$ and $c_0$ such that
for any $n>1$ one has 
\begin{equation}\label{u-n-large-near-0}
u_n(x) \geq c_0 n^{2/d}, \qquad  x\in Z_{\alpha_0 n^{1/d}} .
\end{equation}
Indeed, by Lemma \ref{Lem-balls-inside-out} and DMP we have $u_n (x) \geq n G_{c_1 n^{1/d} } (x)$
for some constant $c_1=c_1(d)$, where the Green's function $G$ is defined in (\ref{Green-def}).
Thanks to this bound, \eqref{u-n-large-near-0} follows directly from asymptotics of Green's function proved in \cite[Proposition 1.5.9]{Lawler-book-walks} for $d\geq 3$,
and \cite[Proposition 1.6.7]{Lawler-book-walks} for $d=2$.
Next, we reduce the general case of the lemma to
\eqref{u-n-large-near-0} by constructing a Harnack chain leading from a given point to the ball $Z_{r_0 n^{1/d}}$.

Let $x^{(0)}= (a_1,a_2,...,a_d)\in \Z^d$ and assume $x^{(0)} \notin Z_{\alpha_0 n^{1/d}}$, since otherwise
the estimate follows by \eqref{u-n-large-near-0}.
Suppose $|a_d| = \max_{1\leq i \leq d} |a_i|$, and  
denote $\beta_0 : = \min\{ \alpha_0/4, r_0/4 \}$, where $\alpha_0$ is fixed from (\ref{u-n-large-near-0}) and $r_0$ from the formulation of the lemma.
We get that the ball
$$
Z_0: = \{x\in \Z^d:  \ |x-x^{(0)}|\leq \beta_0 n^{1/d} \} 
$$
lies inside $V_n \setminus \{0\}$.
By $\Pi_d$ denote the orthogonal projection of $Z_0$ onto $\Z^{d-1}\times \{0\}$, i.e.
$$
\Pi_d  := \{ (x_1,...,x_{d-1},0)\in \Z^d : \ \exists  a\in \Z \text{ s.t. } (x_1,...,x_{d-1},a)\in Z_0 \}.
$$

By directional monotonicity of Theorem \ref{Thm-monotonicity} applied to the direction $e_d$, and axial symmetry of the odometer given by Corollary \ref{cor-O-symm}, 
we have 
\begin{equation}\label{cylinder}
\mathcal{C}_d :=\big( \Pi_d \times [-a_d,a_d]  \big) \cap \Z^d \subset V_n.
\end{equation}
If the cylinder $\mathcal{C}_d $ intersects the ball $Z_{\alpha_0 n^{1/d}}$, the estimate of the lemma follows
by Harnack's inequality of Theorem \ref{Thm-Harnack} and upper estimate of Lemma \ref{Lem-balls-inside-out}, where the latter is used to estimate
the length of Harnack chain. 
Otherwise, if $\mathcal{C}_d \cap Z_{\alpha_0 n^{1/d}} =\emptyset$, denote $x^{(1)} = (a_1,...,a_{d-1}, 0)$,
and $Z_1 = \{x\in \Z^d:  \ |x-x^{(1)}|\leq \beta_0 n^{1/d} \} $.
By construction $Z_1\subset \mathcal{C}_d\subset V_n$.
We have, due to Harnack's inequality, that $u_n(x^{(1)}) \geq C u_n(x^{(0)})$ with a constant $C=C(d, r_0)$,
hence it is enough to prove the lemma for $x^{(0)}$ replaced by $x^{(1)}$.
For that, we apply the same argument as we had for $x^{(0)}$ to $x^{(1)}$.
Since $x^{(1)}$ has at most $d-1$ non-zero coordinates, this reduction procedure on coordinates will terminate in at most $d-1$ number of steps,
where in the last step, the corresponding cylinder (\ref{cylinder}) will intersect the ball $Z_{\alpha_0 n^{1/d} }$
implying the desired estimate of the lemma.
The proof is complete. 
\end{proof}

\section{Shape analysis}\label{sec-Shape-analysis}

The purpose of this section is the study of scaling limit of the model
when the initial distribution is concentrated at a point.
We will show that after proper scaling of mass and the underlying lattice,
there is a convergence along subsequences. The approach is via uniform gradient bounds
for the odometers.
Combined with some results of the previous section,
we will prove a certain non-degeneracy result of the converging shapes, as well as Lipschitz regularity of the boundary of scaling limit.

\vspace{0.2cm}
Throughout  this section,
we assume that initial distribution of the sandpile equals $n\delta_0$, where $n>1$ and boundary capacity
of the model is $ n^{1/d}$. For all $n>1$, by $u_n$ we denote the corresponding odometer function and by $V_n$ the set of visited sites.

\subsection{Gradient bounds}
First, we establish uniform Lipschitz bounds on the odometer functions away from the origin.

\begin{lem}\label{Lem-grad-outside-0}
Let $r_0>0$ be any fixed number.
Then, there exists a constant $C=C(r_0)$ such that for any $n>1$ one has
$$
|u_n(x) - u_n(y)| \leq C n^{1/d},
$$
where $x,y\in V_n$ satisfy $ r_0 n^{1/d} \leq |x| \leq 2r_0 n^{1/d}$ and $x\sim y$.
\end{lem}

\begin{proof}
Let $G_n(x,y)$ be the Green's function for $V_n$,
i.e. $\Delta_x G_n(x,y) = -\delta_0(y-x)$ for any $x,y\in \INT{V_n}$
and $G(x,y)=0$ if $x\in \partial V_n$ or $y\in \partial V_n$.
From \cite[Proposition 4.6.2]{Lawler-book} for $x,y\in \INT{V_n}$ we have
$$
G_n (x,y) = \mathbb{E}^x[ g(S_\tau, y)  ] - g(x,y), 
$$
where as before $g$ is defined from (\ref{g-x-y}), and $\tau$ is the first exit time from $\INT{V_n}$ of
the simple random walk $\{S_k\}_{k=0}^\infty$ started at $x$.
From the definition of $G_n $ we have
\begin{multline}\label{a1}
| G_n(0,x) - G_n(0,y) | \leq |g(0,x) - g(0,y)| + \mathbb{E}^x | g(S_\tau, x) - g(S_\tau, y)  | \leq \\
C \big( |x|^{1-d} + n^{-1} + \max\limits_{w\in \partial V_n} |g(w, x) - g(w, y)| \big) ,
\end{multline}
where $C$ is a dimension dependent constant, and
we have used the asymptotics (\ref{g-x-y}) and Mean-Value Theorem along with the choice of $x$ and $y$ to
bound the first summand.
From Lemma \ref{Lem-balls-inside-out} and the choice of $x,y$ we get that
$|x-w|\geq c_d n^{1/d}$ for any $w\in \partial V_n$.
This, coupled with (\ref{g-x-y}) and Mean-Value Theorem, for any $w\in \partial V_n$ implies
$$
|g(w, x) - g(w, y)| \leq C \frac{|x-y|}{|w-x|^{d-1}}  +  C n^{-1} \leq C n^{-1 + 1/d}.
$$
Combining this estimate with (\ref{a1})
we obtain
\begin{equation}\label{a2}
| G_n(0,x) - G_n(0,y) | \leq C n^{-1+1/d}.
\end{equation}
Since $G_n(x,y) = G_n(y,x)$ we get from the definition of $u_n$ that
$u_n (x) = n G_n(0,x)$ for all $x\in V_n$, which together with (\ref{a2})
completes the proof of the lemma.
\end{proof}

\begin{prop}\label{Lem-Lipschitz-est}{\normalfont{(Uniform Lipschitz bound)}}
For any $r_0>0$ there exists a constant $C=C(r_0)$ such that
for any $n>1$ and all $x,y \in \Z^d$ satisfying $|x|, |y| >r_0 n^{1/d}$ one has
$$
| u_n(x) - u_n(y) | \leq C n^{1/d}|x-y|.
$$
\end{prop}
\begin{proof}
Recall that $V_n\subset \Z^d$ is the set of visited sites for mass $n>1$.
Define 
$$
V_{n,0} = \{ x \in V_n: \ |x|\geq r_0 n^{1/d} \text{ and } \mathrm{dist}(x, \partial V_n) \geq 2  \},
$$
where $\mathrm{dist}$ is the combinatorial distance.
For each $1\leq i \leq d$ consider the discrete derivative of $u$, namely $\partial_i^+ u (x) := u(x+e_i) - u(x) $ with $x\in \Z^d$.
Clearly $\partial_i^+ u$ is harmonic in $V_{n,0}$.
Now, in view of the stability of the sandpile, for any $x\in V_n$ such that $\mathrm{dist}(x, \partial V_n) \leq 1$
we have $|\partial_i^+ u (x) | \leq C_1 n^{1/d}$ with a constant $C_1=C_1(d)$.
On the other hand, by Lemma \ref{Lem-grad-outside-0} we have $|\partial_i^+ u (x)| \leq C_{r_0} n^{1/d}$ if 
$r_0 n^{1/d} \leq  |x| \leq 2 r_0 n^{1/d}$.
Since $\Delta(\partial_i^+ u) =0$ in the interior of $V_{n,0}$, by DMP we get
$|\partial_i^+ u (x)| \leq C n^{1/d}$ in $V_{n,0}$.
The same bound obviously works for $\partial_i^- u(x)  := u(x-e_i) - u(x)$.
We have thus proved 1-step Lipschitz bound, i.e. the estimate of the proposition
if $x$ and $y$ are lattice neighbours.

For the general case, take any $x,y$ such that $|x|, |y| >r_0 n^{1/d}$, and consider the shortest lattice path
connecting $x$ and $y$ and staying outside the ball $B_{r_0 n^{1/d}}$.
Namely, let 
$$
x= X_0 \sim X_1 \sim ... \sim X_k =y,
$$
where $X_i \in \Z^d $ and $|X_i| >r_0 n^{1/d}$. Clearly such path exists.
It is also clear that for the length of the path we have 
$k\approx |x-y|$ where equivalence is with dimension dependent constants. 
Using this and the 1-step Lipschitz bound already proved above, we obtain 
$$
|u(x) - u(y) | \leq \sum\limits_{i=0}^{k-1} | u(X_{i+1} )  - u(X_i)  | \leq C_{r_0} n^{1/d} k \leq C_{r_0} n^{1/d} |x-y|,
$$
completing the proof of the proposition.
\end{proof}

\subsection{Scaled odometers}\label{subsec-scaled-od}
For $n\geq 1$ set $h= n^{-1/d} $, and define the scaled odometer by
$u_h (x) =h^2 u_n(h^{-1} x) $ where $x\in h \Z^d$.
Let also $V_h := h V_n \subset h\Z^d$ be the scaled set of visited sites for mass $n$.
Clearly, $ u_h$ is supported in the interior of $V_h$. Moreover, in view of Lemma \ref{Lem-balls-inside-out}
we have that the sets $\{V_h\}_{0<h\leq 1}$, are uniformly bounded and contain a ball of some fixed radius.

We will need a few notation.
For $0<h \leq 1$ and $\xi =(\xi_1,...,\xi_d)\in h\Z^d$ define the half-open cube
\begin{equation}\label{cube-def}
\mathrm{C}_h(\xi) = \left[ \xi_1 - \frac{h}{2}, \xi_1 +\frac{h}{2} \right)\times ... \times \left[ \xi_d - \frac{h}{2}, \xi_d +\frac{h}{2} \right),
\end{equation}
and for a given set $V\subset h\Z^d$ define $V^{\square} = \bigcup\limits_{\xi\in V} \mathrm{C}_h(\xi)$.

In order to study the scaling limit of the model,
we need to extend each $u_h$ to a function defined on $\R^d$.
We will use a standard extension of $u_h$ which preserves its $\Delta^h$-Laplacian.
Namely, for fixed $0<h\leq 1$ define a function $u_h :\R^d \to \R_+$, where for each $\xi \in h \Z^d$ and any $x\in \mathrm{C}_h(\xi)$ we have set $u_h(x) = u_h(\xi)$. Clearly $\Delta^h u_h (x) = \Delta^h u_h(\xi)$
for all $x\in \mathrm{C}_h(\xi)$.
Note that we are using the same notation for extended odometers.
In what follows $u_h$ will stand for this extension.
Also, for a given set $E\subset \R^d$ we write $\INT{E}$
for the interior of $E$ in the next theorem.
The following is our main result concerning scaling limit of the model.

\begin{theorem}\label{Thm-scaling-limit}
There exists a sequence $h_k\to 0$, and compactly supported non-negative function $u_0\in C(\R^d\setminus \{0\})$
which is Lipschitz outside any neighbourhood of the origin, such that
\begin{itemize}
 \item[{\normalfont{(i)}}] $ u_{h_k} \to u_0 $ locally uniformly in $\R^d\setminus \{0\}$,
 \vspace{0.2cm}
 \item[{\normalfont{(ii)}}] $\Delta u_0 = -\delta_0$ in $\{u_0>0\}$
 in the sense of distributions, where $\delta_0$ is the Dirac delta at the origin
 and $\Delta$ denotes the continuous Laplace operator,
\item[{\normalfont{(iii)}}] if $V_0\subset \R^d$ is the support of $u_0$ then,
 $ V_0 = \bigcup\limits_{ m=1 }^\infty  \bigcap\limits_{k=m}^\infty V_{h_k}^\square, $
 \vspace{0.2cm}
\item[{\normalfont{(iv)}}] if we set $u_0(0)=+\infty$, and let the set of vectors
$\mathcal{N}$ be defined as in \eqref{dir-N},
then for any $x^{(1)}, x^{(2)} \in \R^d$ with the property that
 $x^{(2)}-x^{(1)}$ is non-zero and is collinear to any of the vectors of $\mathcal{N}$,
we have
$$
u_0(x^{(1)}) \geq u_0(x^{(2)}) \ \ \text{ if } \ \ |x^{(1)}| \leq |x^{(2)}|.
$$

 \vspace{0.2cm}
 \item[{\normalfont{(v)}}] the boundary of $V_0 $ is locally a Lipschitz graph.
\end{itemize} 
\end{theorem}

\begin{proof}
Fix $\rho>0$ small, and for $0<h \leq 1$ set 
$$
E_{\rho}(h)  = \{ x \in h\Z^d: \ |x| > \rho \}.
$$
Extend each $u_h$ as 0 outside $V_h$.
From Proposition \ref{Lem-Lipschitz-est} and definition of $u_h$, there exists a constant $C_\rho=C(\rho,d )$
independent of $h$
such that 
\begin{equation}\label{u-h-Lipschitz}
|u_h(x) - u_h(y)| \leq C_\rho |x-y|, \qquad x,y\in E_{\rho}(h).
\end{equation}
For $y\in \partial V_h$ we have  $u_h(y ) =0$, hence for any $x\in E_\rho (h) \cap V_h$ thanks to (\ref{u-h-Lipschitz}) we obtain
\begin{equation}\label{u-h-unif-bound}
  |u_h(x)| \leq C_\rho  |x| \leq C_\rho C_1,
\end{equation}
where the second inequality with a constant $C_1=C_1(d)$ follows from Lemma \ref{Lem-balls-inside-out}.
Thus each $u_h$, $0<h\leq 1$, when restricted to $E_\rho(h)$, 
is bounded uniformly in $h$ due to \eqref{u-h-unif-bound} and is $C_\rho$-Lipschitz
in view of \eqref{u-h-Lipschitz}.
Next, we extend each $u_h$ from $E_\rho(h)$ to a Lipschitz function on $\R^d$ having the same Lipschitz constant.
Namely, for a given $h>0$ define
\begin{equation}
 U_h^\rho(x) = \inf\limits_{\xi \in E_{\rho}(h)} \big( u_h(\xi) + C_\rho |x-\xi| \big) , \qquad x\in \R^d. 
\end{equation}
This is a well-known method of Lipschitz extension, and it is not hard to verify that $U_h^\rho$ is a $C_\rho$-Lipschitz function on $\R^d$
(see e.g. \cite[Theorem 2.3]{Heinonen}) and coincides with $u_h$ on $E_\rho(h)$.
Moreover, by construction we have that the family $\{U_h^\rho\}_{0<h\leq 1}$ is uniformly bounded and
is non-negative everywhere. Observe also that for $0<h\leq 1$, $\xi \in E_{2 \rho }(h)$
and $x \in \mathrm{C}_h(\xi)$ (see \eqref{cube-def}) by construction we have
\begin{equation}\label{u-h-U-h-est}
|u_h(x) - U_h^\rho(x)| = |U_h^\rho(\xi) - U_h^\rho(x)| \leq C_\rho h.
\end{equation}

By Arzel\`{a}-Ascoli  there is sequence $h_k\to 0$ as $k\to \infty$, and $C_\rho$-Lipschitz function
$U_0^\rho:\R^d  \to \R$ such that
$U_{h_k}^\rho \to U_0^\rho$ locally uniformly in $\R^d$. 
Due to (\ref{u-h-U-h-est}) we get $u_{h_k} \to U_0^\rho$ locally uniformly in $E_{2\rho}^\square$.
Obviously $U_0^\rho \geq 0$ on $\R^d$. 

\vspace{0.2cm}

For the support of $U_0^\rho$ we have
\begin{equation}\label{supp-outside-2delta}
\{ x\in \R^d: \ U_0^\rho(x) >0 \} \setminus \overline{B_{2\rho}} =
\left( \liminf\limits_{k\to \infty} V_{h_k}^\square  \right)^\circ \setminus \overline{B_{2\rho}}.
\end{equation}
To prove this take any $x$ from the \emph{l.h.s.} of (\ref{supp-outside-2delta}).
Since $U_0^\rho$ is continuous on $\R^d$ there is $r>0$ such that
the closed ball $\overline{B_r}(x) \subset \R^d \setminus \overline{B_{2\rho}} $
and $U_0^\rho > \e$ on $\overline{B_r}(x)$ for some $\e>0$.
By construction we have that $U_{h_k}^\rho$ converges to $U_0^\rho$
uniformly on $\overline{B_r}(x)$, hence there exists $k_0 \in \N$ large enough
such that $U_{h_k}^\rho (x) \geq \e$ on $\overline{B_r}(x)$ for any integer $k>k_0$.
But since $U_{h_k}^\rho$ agrees with $u_{h_k}$ on $E_\rho (h_k)$
we get $u_{h_k}(\xi) \geq \e$ for any $\xi \in h_k \Z^d \cap \overline{B_{r/2}}(x)$.
This implies $u_{h_k}(x) \geq \e$ on $(h_k \Z^d \cap \overline{B_{r/2}}(x) )^\square$
hence $\overline{B_{r/2}}(x) \subset V_{h_k}^\square \setminus \overline{B_{2\rho}}$
for all $k>k_0$. In particular, $x$ is from the \emph{r.h.s.} of (\ref{supp-outside-2delta}).

To see the reverse inclusion
fix any $x$ from the \emph{r.h.s.} of (\ref{supp-outside-2delta}). There exists $k_0\in \N$ large and $r>0$ small such that 
$B_r(x) \subset \bigcap\limits_{k=k_0}^\infty V_{h_k}^\square \setminus \overline{B_{2\rho}}$.
We get, in particular, that $\mathrm{dist}(x, \partial V_{h_k}^\square) \geq r$
for all $k>k_0$. Hence, by Lemma \ref{Lem-bounds-below} and definition of $u_{h_k}$
one gets $u_{h_k} \geq c_0>0$ on $B_{r}(x)$ with a constant $c_0$ 
uniform in $k>k_0$. This bound, together with (\ref{u-h-U-h-est}) implies
$U_0^\rho(x)\geq c_0$, accordingly $x$ lies in the \emph{l.h.s.} of (\ref{supp-outside-2delta}).
This completes the proof of (\ref{supp-outside-2delta}).

\vspace{0.2cm}

We next prove that $U_0^\rho$ is harmonic on $\{x\in \R^d:  U_0^\rho(x)>0\} \setminus \overline{B_{2\rho}} $.
Fix any $x\in \R^d\setminus \overline{B_{2\rho} }$ such that $U_0^\rho(x)>0$. Since $U_0^\rho$ is continuous,
there exists a closed ball $\overline{B_{2r}}(x)\subset \{U_0^\rho>0 \} \setminus \overline{B_{2\rho}}$. Let $\varphi\in C_0^\infty(B_r(x) )$
be any.
Since $U_0^\rho$ is continuous, to prove its harmonicity in $B_r(x)$ it suffices, due to Weyl's lemma,
to show that $\int_{\R^d} U_0^\rho \Delta \varphi dx=0$ where $\Delta$ is the usual (continuous) Laplace operator in $\R^d$.
The latter follows by
$$
\int\limits_{\R^d} U_0^\rho \Delta \varphi dx = 
\lim_{k \to \infty} \int\limits_{\R^d} U_{h_k}^\rho \Delta^{h_k} \varphi dx = 
\lim_{k \to \infty} \int\limits_{\R^d} u_{h_k} \Delta^{h_k} \varphi dx  = 
\lim_{k \to \infty} \int\limits_{\R^d} \Delta^{h_k} u_{h_k} \varphi dx =0,
$$
where the first equality comes from definition of $U_0^\rho$ and smoothness of $\varphi$,
the second is in view of (\ref{u-h-U-h-est}),
the third one follows from discrete integration by parts. The last equality
is a consequence of the fact that for large enough $k_0\in \N$ due to \eqref{supp-outside-2delta} one has $B_{2r}(x) \subset V_{h_k}^\square \setminus \overline{B_{2\rho}} $
for all $k>k_0$, which, in particular, implies that $\Delta^{h_k} u_{h_k} =0 $ on $B_{r}(x) $ for all $k>k_0$.

Taking the parameter $\rho \to 0$ and applying a diagonal argument
we conclude that there is a non-negative function $u_0\in C(\R^d\setminus \{0\})$, and a sequence $h_k\to 0$ such that
\begin{itemize}
 \item[(a)] $u_{h_k} \to u_0$ locally uniformly outside any neighbourhood of 0,
 \vspace{0.1cm}
 \item[(b)] $\supp u_0 = \liminf\limits_{k\to \infty}  V_{h_k}^\square $,
 \vspace{0.1cm}
 \item[(c)] $\Delta u_0  = 0$ in $\{x\in \R^d\setminus \{0\}: \ u_0(x) >0 \}$,
 \vspace{0.1cm}
 \item[(d)] $u_0$ is Lipschitz continuous outside any neighbourhood of the origin,
 where Lipschitz norm depends on the neighbourhood.
\end{itemize}

These properties give us parts (i) and (iii) of the theorem, as well as the claim of (ii) except at the origin.
To finish the proof of part (ii) it remains to study the Laplacian of $u_0$ at 0.
To this end, let $\Phi_h$ be the fundamental solution for $\Delta^h$ in $h\Z^d$, i.e.
$\Delta^h \Phi_h = -h^d \delta_0$. Here as well, following the convention above, 
suppose that $\Phi_h$ is extended to $\R^d$ so that to preserve its $\Delta^h$-Laplacian.
It is well-known (see \cite{Lawler-book}) that $\Phi_h \to \Phi_0$ as $h\to 0$ locally uniformly in $\R^d\setminus \{0\}$,
where $\Phi_0$ is the fundamental solution to continuous Laplacian.
Observe, that in view of (b) and Lemma \ref{Lem-balls-inside-out} there is $r>0$ small such that the ball $\overline{B}_r \subset \supp u_0$.
Using this and the definition of $u_{h_k}$ we have
$$
\Delta^{h_k} ( u_{h_k}  - \Phi_{h_k}) = 0 \text{ in } B_r.
$$
But since both $u_{h_k}$ and $\Phi_{h_k}$ converge locally uniformly away from 0 to continuous functions,
then the difference $|u_{h_k}  - \Phi_{h_k}| \leq C $ on $\partial B_r$ for all $k\in \N$ where $C>0$ is some large constant.
By DMP we get that $|u_{h_k}  - \Phi_{h_k}| \leq C $ in $B_r$,
hence, the limit $|u_0 - \Phi_0|$ is also bounded by $C$ in $B_{r}\setminus \{0\}$.
But as $\Delta(u_0 - \Phi_0) = 0 $ in $B_{r}\setminus \{0\}$ we get that $0$ is a removable singularity for $u_0 - \Phi_0$,
in particular $\Delta u_0 = \Delta \Phi_0 $ at 0, and the proof of part (ii) is now complete.

\vspace{0.2cm} 

The proof of part (iv) of the theorem is a direct consequence of  Theorem \ref{Thm-monotonicity}, and convergence of $u_{h_k}$ to $u_0$.

\vspace{0.2cm} 

It remains to show (v), the last assertion of the Theorem.
We will show that at each $x^{(0)}\in \partial V_0$ there exists a double cone,
with its size (opening and height) independent of $x^{(0)}$, having
vertex at $x^{(0)}$ and intersecting $\partial V_0 $ at $x^{(0)}$
only\footnote{Observe that $\partial V_0 $ is the 0-level set of $u_0$ which is a Lipschitz function on compact subsets
away from the origin, and in particular in a neighbourhood of $\partial V_0$.
But this fact alone is not enough to conclude that $\partial V_0$ is Lipschitz,
see for instance \cite{ABC}.}. This clearly implies (v).

Applying the monotonicity of (iv) in directions $\{e_i\}_{i=1}^d$ we obtain $u_0(x) = u_0(-x)$ for all $x\in \R^d$.
Due to this symmetry, to prove (v) it is enough to treat the part of $\partial V_0$ where all coordinates are non-negative.
In addition, we will also assume that the point $x^{(0)} = (x_1^{(0)}, ..., x_d^{(0)} ) \in \partial V_0$
satisfies $x_d^{(0)} \geq x_i^{(0)}$ for all $ 1 \leq i \leq d-1$. The rest of the cases are similar.
Observe that due to (iii) and Lemma \ref{Lem-balls-inside-out} there is a constant $c_0>0$ such that
\begin{equation}\label{ball-inside}
 B_{c_0} \subset \{ x \in \R^d:   u_0>0 \}.
\end{equation}
Now the choice of $x^{(0)}$ and (\ref{ball-inside}) together imply
\begin{equation}\label{x-0-d-large}
 x_d^{(0)} \geq \frac{1}{d} c_0.  
\end{equation}

We first determine a certain monotonicity region for $x^{(0)}$. Let $\mathcal{V} \subset \R^d$ be a finite set of vectors containing $e_d$
and having the rest of its elements chosen by the following rule:
for each $1\leq i \leq d-1$ 
\begin{itemize}
 \item[($\star$)] if $x_d^{(0)} \geq 2 x_i^{(0)} $ then $ e_d + e_i   $ and $e_d - e_i$ are in $\mathcal{V}$,
 \vspace{0.2cm}
 \item[($\star \star $)] if $x_d^{(0)} \leq 3 x_i^{(0)} $ then $e_i $ is in $\mathcal{V}$.
\end{itemize}

\vspace{0.1cm}

It is easy to check from definition of $\mathcal{V}$ that it contains $d$ linearly independent vectors.
For a given $v\in \mathcal{V}$ consider the halfspace $ \mathcal{H}_v = \{ x\in \R^d: \ x\cdot v>0  \}$.
We have the following monotonicity of $u_0$ in $\mathcal{H}_v$. 

\begin{equation}\label{mono-plane}
\text{If } X_1, X_2 \in \mathcal{H}_v, \   X_2-  X_1   = t v  \text{ with } t\in \R, \text{ then }
u_0(X_1) \geq u_0(X_2) \text{ iff } t\geq 0. 
\end{equation}
Indeed, decomposing $X_i$ into tangential and normal components in $\mathcal{H}_v$, one gets
$ X_i = \tau_0  + t_i v $, where $i=1,2$, $ t_i \geq 0 $ and $\tau_0\in \mathcal{H}_v$ is the same for both $X_1, X_2$ due to assumption.
Then, $|X_i |^2  = |\tau_0 |^2 + t_i^2 |v|^2  $, hence
$|X_1| \leq |X_2|$ if and only if $|t_1| \leq |t_2|$ but since $t_i\geq 0$ the latter reduces to $t_1\leq t_2$.
It remains to apply the monotonicity result of part (iv) to get \eqref{mono-plane}.
We thus obtain that $u_0$ is non-increasing in $\mathcal{H}_v$ in the normal direction.

Consider the cone
$$
\mathcal{C}_0 = \bigcap\limits_{v\in \mathcal{V}} \mathcal{H}_v.
$$
It is easy to see that the bound \eqref{x-0-d-large} and definition of $\mathcal{V}$ imply the existence of a constant $c_1>0$ independent of $x^{(0)}$ such that 
\begin{equation}\label{x-0-is-inside}
B_{c_1}(x^{(0)}) \subset \mathcal{C}_0. 
\end{equation}

As a direct corollary of (\ref{mono-plane}) the cone $\mathcal{C}$ inherits the following property.
If $X_1, X_2 \in \mathcal{C}_0$ and $X_2 - X_1 = t_1 v_1 +...+t_k v_k$ where $t_i \geq 0$, $v_i \in \mathcal{V}$ for all $1\leq i\leq k$, and $k $ is the cardinality of $\mathcal{V}$,
then
\begin{equation}\label{mono-2}
u_0(X_1 ) \geq u_0(X_2).
\end{equation}

Along with $\mathcal{C}_0$ consider as well the cone generated by $\mathcal{V}$
$$
\mathcal{C}_V = \{ t_1 v_1 +...+t_k v_k : \ t_i \geq 0, v_i \in \mathcal{V}, \text{ for all } 1\leq i \leq k \text{ and }  k=|\mathcal{V}| \}.
$$
Since $\mathcal{V}$ has $d$ linearly independent vectors, the cone $\mathcal{C}_V$ is $d$-dimensional.
We claim that $\mathcal{C} : = ( x^{(0)} - \mathcal{C}_V  ) \cap \mathcal{C}_0 $ is the sought cone,
for which we need to show that the truncated cone $\mathcal{C}$ has uniform size
and that $u_0$ is positive in the interior of $\mathcal{C}$ and is zero in the interior of 
$( x^{(0)} + \mathcal{C}_V  ) \cap \mathcal{C}_0 $. We will only show the former claim, as the latter follows similarly.

It is clear, thanks to (\ref{x-0-is-inside}), that $\mathcal{C}$ is of uniform size. 
Now assume for contradiction, that there is $x^\ast \in \INT{\mathcal{C}}$ such that $u_0(x^*) =0$. 
Consider the cone $\mathcal{C}^* = ( x^\ast + \mathcal{C}_V ) \cap \mathcal{C}_0 $.
Let us see that
\begin{equation}\label{2claims}
  u_0(x^*) \geq u_0(x), \ \forall x\in \mathcal{C}^* \qquad \text{and} \qquad x^{(0)} \text{ is in the interior of } \mathcal{C}^* .
\end{equation}
The first assertion of (\ref{2claims}) follows directly from \eqref{mono-2} since $x^* \in \mathcal{C}_0$ by definition
and any element $x\in \mathcal{C}^*$ is from $\mathcal{C}_0$ as well,
and has the form $x = x^* + t_1 v_1+...+t_k v_k$ with all $t_i \geq 0$.
The second one is simply a consequence of the definition of $x^*$.

Armed with (\ref{2claims}) the proof of (v) follows readily, since we simply get that $u_0$ is zero in an open neighbourhood of $x^{(0)}$,
which violates the condition that $x^{(0)} \in \partial V_0$. This contradiction
finishes the proof of (v), and the proof of the theorem is now complete.
\end{proof}

\begin{proof}[Proof of Theorem \ref{Thm-ASM}]
Assertion (a) follows from convergence result of Pegden-Smart
discussed in subsection \ref{subsec-ASM} in conjunction with Theorem \ref{Thm-monotonicity-ASM}.
The proof of part (b) follows from part (a) by the same argument as for part (v) of Theorem \ref{Thm-scaling-limit}.
The proof is complete.
\end{proof}

\subsection{Concluding remarks and open problems}
We finish the paper with a few remarks and open problems which we think are interesting.

\vspace{0.2cm}

\begin{remark}{\normalfont{(On monotonicity)}}
Observe that monotonicity of the odometer proved in Theorems \ref{Thm-monotonicity}, \ref{Thm-monotonicity-ASM},
\ref{Thm-ASM}, and \ref{Thm-scaling-limit} is not strict. Namely, we expect 
the odometer to be strictly monotone on its support. Such statement can be proved
for the boundary sandpile model relying on harmonicity of odometer and the strong maximum
principle, however for the classical Abelian sandpile model (ASM) it seems to require a more careful analysis, which we leave open.
We also think that the ``Boundary Abelian sandpile'' described below in the text following 
 Problem 2,
should enjoy the monotonicity property of the odometer. But to actually prove that,
one needs to work out the analogue of the least action principle (see subsection \ref{subsec-ASM}).

Next, we remark there are a few interesting (and not hard to prove) corollaries of the discrete monotonicity
of the odometer of the ASM proved in Theorem \ref{Thm-monotonicity-ASM}.
First, the monotonicity of the odometer function implies 
that the visited set of the ASM with $n\in \N$ chips at the origin,
is simply connected for any $n$, meaning that its complement in $\Z^d$ is connected as a graph.
This recovers the result of Fey and Redig \cite[Proposition 4.12]{Fey-Redig}
by a different method.
Next, relying on monotonicity of the odometer, one can show that the visited set of the ASM contains 
a diamond (a convex hull of the points $\{\pm e_i, i=1,...,d\}$ restricted to $\Z^d$) scaled by a factor of $R \approx n^{1/d}$,
and is contained in a cube of size $R \approx n^{1/d}$ (cf. \cite[Proposition 3.7]{AS}).
Of course this result is not new, and moreover the constants we can get in these embeddings are slightly worse
than the ones obtained by Fey and Redig \cite{Fey-Redig}, and imporved later by Levine and Peres \cite{Lev-Per}.
What we think is interesting here, is that the method is different,
and seems to be of a general nature, relying mainly on the monotonicity of the odometer.
Hence, similar embeddings should be possible to obtain for the boundary sandpile,
or the boundary abelian sandpile, without much difficulty.

The details of this discussion will be given elsewhere.
\end{remark}

\noindent \textbf{Problem 1.} Does the boundary sandpile process, for initial distributions concentrated at a point, have a scaling limit?

\vspace{0.2cm}

It is clear from the proof of Theorem \ref{Thm-scaling-limit}
that for any sequence of scalings $\mathfrak{h}:=\{h_k\}_{k=1}^\infty$, where $h_k\to  0$,
one may extract a subsequence $\{h_{k_m} \}$ such that the corresponding sequence of scaled odometers
will have a limit, call it $u_{\mathfrak{h}}$, satisfying the same properties as
the function $u_0$ of Theorem \ref{Thm-scaling-limit}.
In particular, by Theorem \ref{Thm-scaling-limit} (ii) we have
that any such $u_{\mathfrak{h}}$ is $C^\infty$ on its positivity set away from the neighbourhood of the origin,
and is Lipschitz continuous up to the boundary of that set.
However, we do not know if scaling limits generated by different
scaling sequences must coincide.

\vspace{0.2cm}

For several lattice-growth models the existence of scaling limit, and in some cases its geometry,
are known; for  divisible sandpile see \cite{Lev-thesis}, \cite{Lev-Per} and \cite{Lev-Per10},
for rotor-router see \cite{Lev-thesis}, \cite{Lev-Per}, \cite{Lev-Per10}, and
\cite{Fey-Redig}, and for Abelian sandpile see \cite{PS}.
The internal diffusion limited aggregation (IDLA), which is a stochastic growth model,
is studied in \cite{LBG}, \cite{L95}, \cite{Am-Al}, \cite{Jer-Sh-Lev}.
For the IDLA in critical regime, one may consult \cite{Grav-Quast}.
Growth models involving continuous amounts of mass similar to divisible
sandpile model, are studied in \cite{From-Jar}, \cite{Lucas}, and \cite{Fey-automata},
where the last paper considers a non-abelian growth model.

\vspace{0.2cm}

\noindent \textbf{Problem 2.} Assume the answer to Problem 1 is yes.
Then, what would be the limiting PDE problem
and what can be said about the geometry of the limiting shape?
In particular, will it be convex?

\vspace{0.2cm}

It is interesting to observe, that the strong flat patterns appearing on the boundary of a sandpile
as can be seen in Figure \ref{Fig-1}, might be a result of a significant
increase of the role of the free boundary, rather than the PDE problem in the interior.
It is remarkable that a similar geometry of the boundary as in Figure \ref{Fig-1},
in particular the flatness, arises also
from the classical Abelian sandpile, see Figure \ref{Fig-2},
if one treats the boundary of the growth cluster likewise,
i.e. allowing each boundary point to accumulate many particles however not exceeding the given threshold.
Nevertheless, the PDE problem of the Abelian sandpile in the set of visited sites,
is very different from that of the boundary sandpile considered here.
It seems very interesting to understand the emergence of the flatness on the boundary
rigorously.
 
\begin{figure}[!htbp]
\centering
\begin{subfigure}
  \centering
  \includegraphics[width=3in]{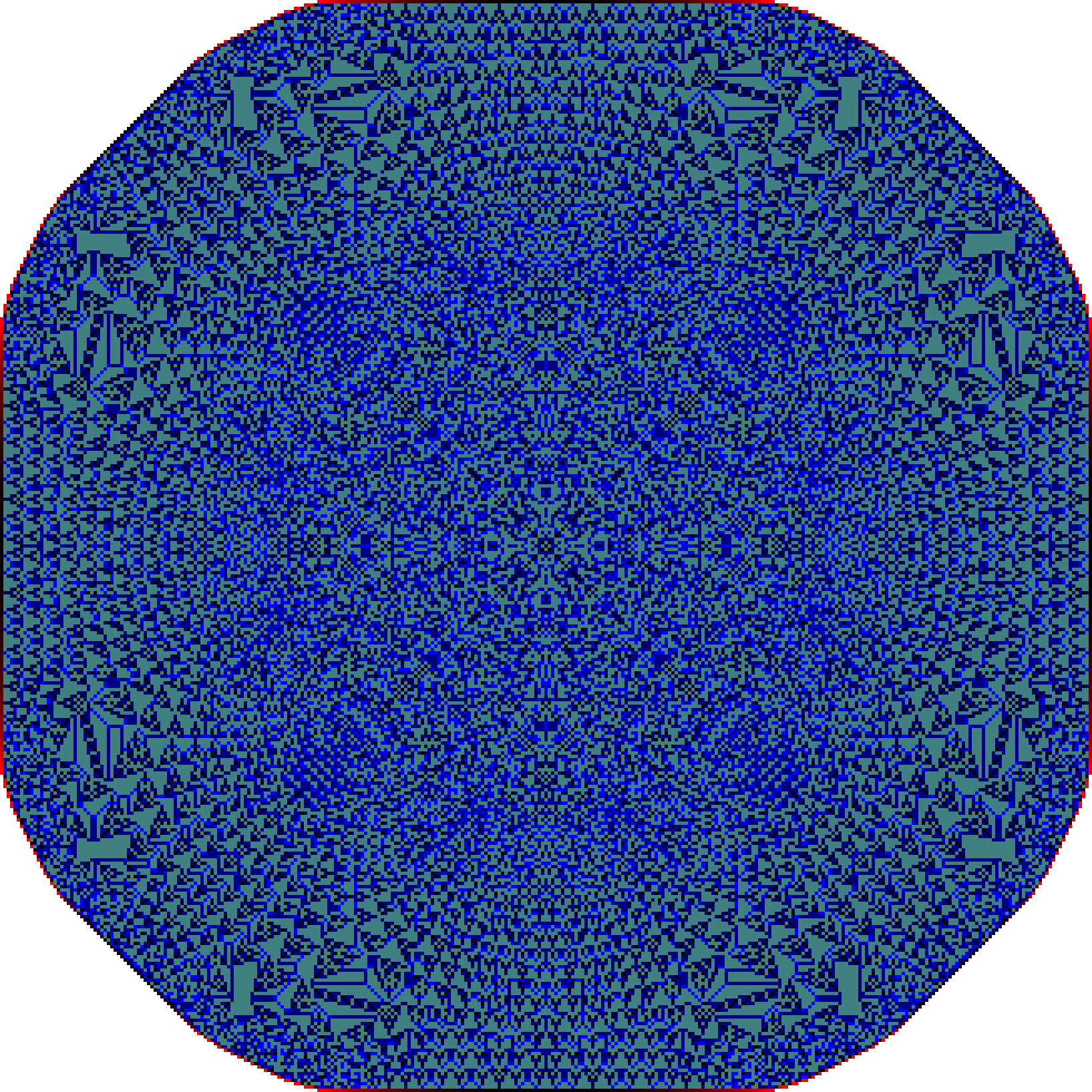}
  \label{fig:sub1}
\end{subfigure}%
\caption{\footnotesize{The classical Abelian Sandpile on $\Z^2$ obtained from one million particles
at the origin and having boundary capacity equal 1 000.
Lattice sites carrying 0,1,2,3 particles are coloured respectively by
black, blue, dark blue, and cyan. The colouring scheme of the boundary is the same
as in Figure \ref{Fig-1}}.}
\label{Fig-2}
\end{figure}

\vspace{0.2cm}

Yet another surprising property of the boundary sandpile is the discontinuity of
its geometry under small perturbations of the initial mass distribution,
and a certain tendency of the sandpile dynamics towards generating convex sets
(see Figure \ref{Fig-3}).
A rigorous explanation of these phenomena seems to be a very interesting problem.
 
\begin{figure}[!htbp]
\centering
\begin{subfigure}
  \centering
  \includegraphics[width=2.2in]{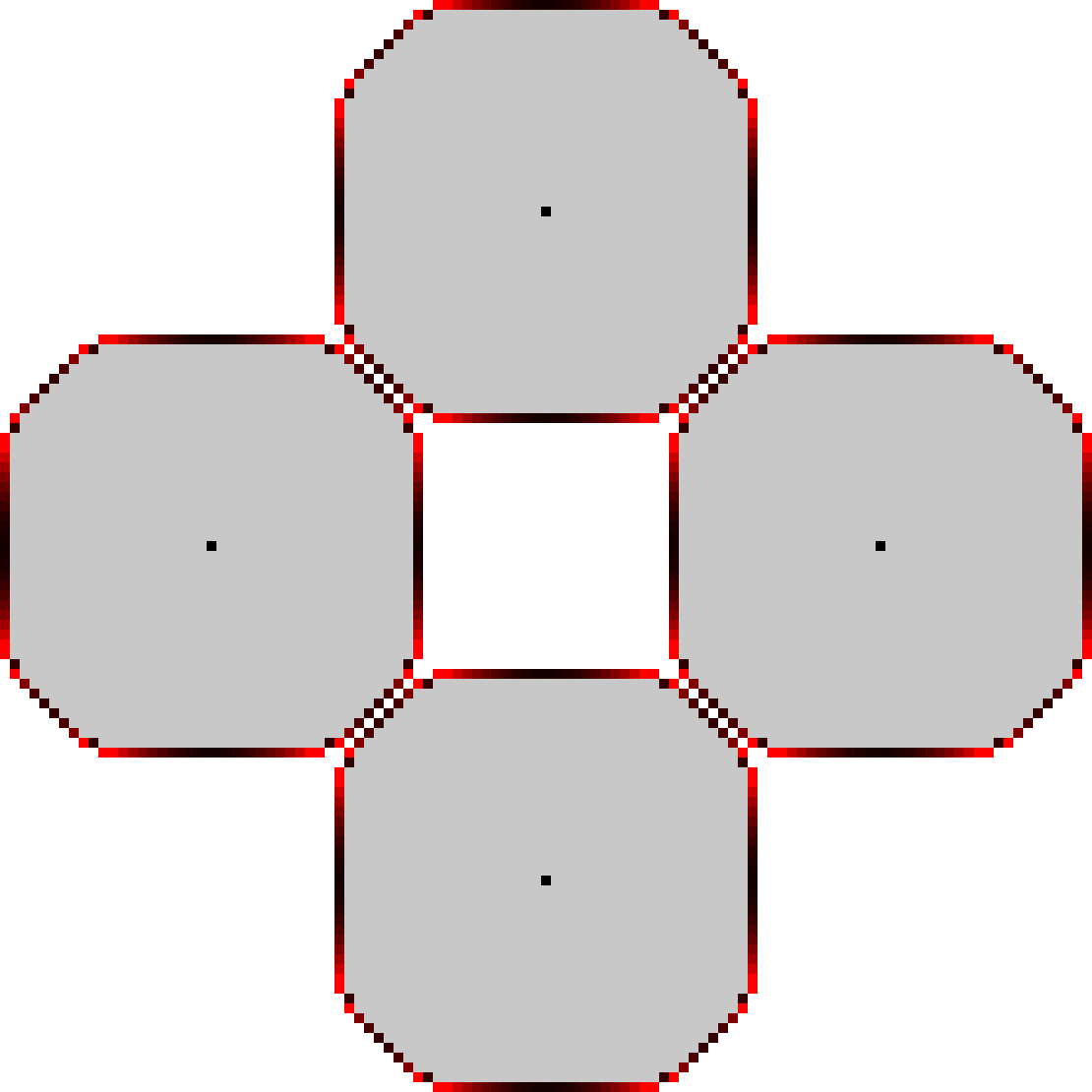}
  \label{fig:sub1}
\end{subfigure}%
\qquad
\begin{subfigure}
  \centering
  \includegraphics[width=2.2in]{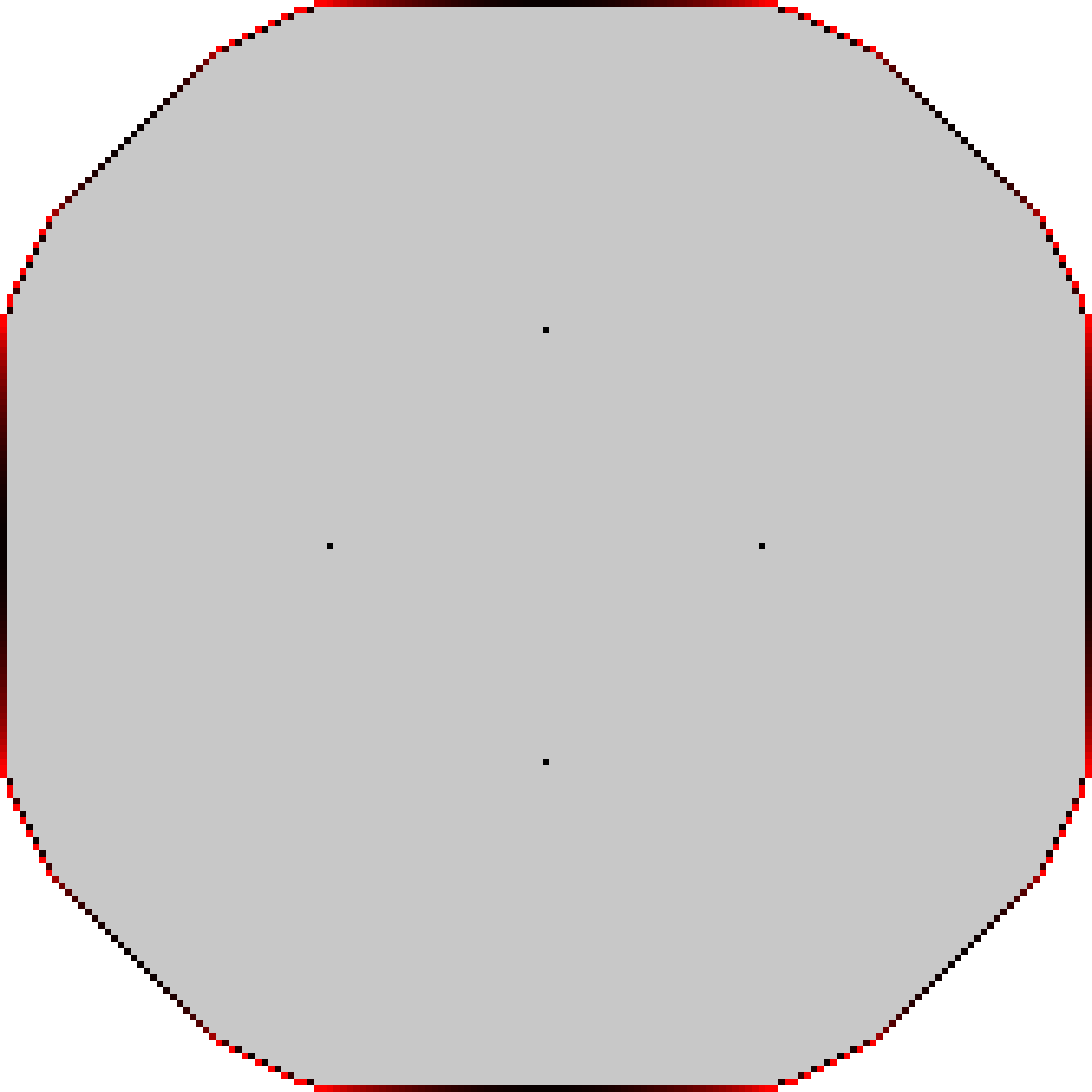}
  \label{fig:sub1}
\end{subfigure}%
\caption{\footnotesize{The left image is a $\mathrm{BS}$ with initial distribution
concentrated at points $(\pm 34, 0)$ and $(0, \pm 34)$ on $\Z^2$, all having mass 40 000,
and the boundary capacity is set to $400$. The clusters
barely survive intersection by one lattice site.
On the right, is a $\mathrm{BS}$ having capacity 400, where mass 40 000
is concentrated on each of the sites $(\pm 33,0)$ and $(0, \pm 33)$.
The colouring scheme is identical to that of Figure \ref{Fig-1}.}}
\label{Fig-3}
\end{figure}

\vspace{0.2cm}

\noindent \textbf{Problem 3.} Is there a boundary sandpile type process,
leading to that the total mass of the system is being redistributed onto the combinatorial free boundary
(possibly with a different background rules) which has a sphere as its scaling limit?

\vspace{0.2cm}

In fact, this problem does not exclude, as a possible candidate,
the boundary sandpile process considered in this paper.
It is apparent from numerical simulations (see Figure \ref{Fig-1}), that the shape of $\mathrm{BS}$ divaricates from a sphere. However, we do not have a rigorous proof of this fact.

\bigskip

Months after the submission of this paper, the current authors found a new sandpile redistribution
rule \cite{AS}, which partially answers the quest of Problem 3, at least for single sources.
The relevance of Problem 3, however, remains in place, as there is no reason to believe that
there does not exist a more natural rule than the one obtained in \cite{AS}.


\begin{thebibliography}{99}

\bibitem{AS} Aleksanyan, H., Shahgholian, H.: Perturbed divisible sandpiles and quadrature surfaces. preprint arXiv:1703.07568 (2017)

\bibitem{Alexandrov} Alexandrov, A.D.: A characteristic property of spheres, Ann. Mat. Pura Appl. (4) 58, 303-315 (1962)

\bibitem{ABC} Alberti, G., Bianchini, S., Crippa, G.:
Structure of level sets and Sard-type properties of Lipschitz maps,
Ann. Sc. Norm. Super. Pisa Cl. Sci. (5), 12(4), 863-902 (2013)

\bibitem{AC} Alt, H. W., and Caffarelli, L. A.: Existence and regularity for a minimum problem with free boundary,
J. Reine Angew. Math. 325, 105-144 (1981)

\bibitem{Am-Al} Asselah, A., Gaudilli\`{e}re, A., From logarithmic to subdiffusive polynomial fluctuations for internal DLA and related growth models,
Ann. Probab. 41(3A), 1115-1159 (2013)

\bibitem{Bak} Bak, P., Tang, C., Wiesenfeld, K.: Self-organized criticality: An explanation of the $1/f$ noise,Phys. Rev. A (3) 38, 364-374 (1988)

\bibitem{BLS} Bj\"{o}rner, A., Lov\'{a}sz, L., Shor, P.:
Chip-firing games on graphs, European J. Combin. 4(12), 283-291 (1991)

\bibitem{QD-book} Ebenfelt, P., et al. (eds.): Quadrature Domains and Their Applications. The Harold S. Shapiro Anniversary Volume. Birkh\"{a}user, Basel (2005)

\bibitem{FLP} Fey, A., Levine, L., Peres, Y.: Growth rates and explosions in sandpiles, J. Stat.
Phys. 138, 143-159 (2010)

\bibitem{Fey-automata} Fey-den Boer, A., Liu, H.: Limiting shapes for a nonabelian sandpile growth model and related cellular automata, 
J. Cell. Autom. 6, 353-383 (2011)

\bibitem{From-Jar} Fr\'{o}meta, S., Jara, M.: Scaling limit for a long-range divisible sandpile. preprint at arXiv:1507.03624 (2015)

\bibitem{Fey-Redig} Fey, A., Redig, F.: Limiting shapes for deterministic centrally seeded growth models. J. Statist.
Phys. 130(3), 579-597 (2008)

\bibitem{Fuk-Uc} Fukai, Y., Uchiyama, K.: Potential kernel for two-dimensional random walk, Ann. Probab. 24(4), 1979-1992 (1996)

\bibitem{Grav-Quast} Gravner, J., Quastel, J.: Internal DLA and the Stefan problem, Ann. Probab. 28(4), 1528-1562 (2000)

\bibitem{Gust-part} Gustafsson, B.: Direct and inverse balayage-some new developments in classical potential theory. 
Proceedings of the Second World Congress of Nonlinear Analysts, Part 5 (Athens, 1996) 
Nonlinear Anal. 30 (1997), no. 5, 2557-2565

\bibitem{GS} Gustafsson, B., Shahgholian, H.: Existence and geometric properties
of solutions of a free boundary problem in potential theory,
J. Reine Angew. Math. 473, 137-179 (1996)


\bibitem{Heinonen} Heinonen, J.: Lectures on Lipschitz analysis.  University of Jyv\"{a}skyl\"{a} (2005)

\bibitem{Jer-Sh-Lev} Jerison, D., Levine, L., Sheffield, S.: Logarithmic fluctuations for internal DLA,
J. Amer. Math. Soc. 25, 271-301  (2012)

\bibitem{LBG} Lawler, G., Bramson, M., Griffeath, D.: Internal diffusion limited aggregation. Ann. Probab.
20(4), 2117-2140 (1992)

\bibitem{L95} Lawler, G.: Subdiffusive fluctuations for internal diffusion limited aggregation. Ann. Probab.
23(1), 71-86 (1995)

\bibitem{Lawler-book-walks} Lawler, G.: Intersections of Random Walks, (Probability and Its Applications) Birkh\"{a}user 1996

\bibitem{Lawler-book} Lawler, G., Limic, V.: Random Walk: A modern introduction, Cambridge Studies in Advanced Mathematics 2010

\bibitem{Lev-Peg-Smart} Levine, L., Pegden, W., Smart, C.K.: Apollonian structure in the Abelian sandpile,
§Geom. Funct. Anal. 26(1) 306-336 (2016)

\bibitem{Lev-Per} Levine, L., Peres, Y.: Strong spherical asymptotics for rotor-router aggregation and the divisible sandpile, Potential Anal. 30(1), pp.1-27 (2009)

\bibitem{Lev-Per10} Levine, L., Peres, Y.: Scaling limits for internal aggregation models with multiple sources, 
J. Anal. Math. 111(1), 151-219 (2010)

\bibitem{Lev-thesis} Levine, L.: Limit Theorems for Internal Aggregation Models, PhD thesis,
University of California Berkley (2007)


\bibitem{Lucas} Lucas, C.: The limiting shape for drifted internal diffusion limited
aggregation is a true heat ball, Probab. Theory Relat. Fields 159, 197-235 (2014) 

\bibitem{PS} Pegden, W., Smart, C. K.: Convergence of the abelian sandpile, 
Duke Math. J. 162(4), 627-642 (2013)

\bibitem{Serrin} Serrin, J.: A symmetry problem in potential theory, Arch. Ration. Mech. Anal. 43, 304-318 (1971)
 

\bibitem{Uchiyama} Uchiyama, K.: Green's functions for random walks on $\Z^N$, Proc. Lond. Math. Soc. 77(1), 215-240 (1998)

\bibitem{Zidarov} Zidarov, D.: Inverse Gravimetric Problem in Geoprospecting and Geodesy. Developments in solid earth geophysics, Elsevier (1990)

\end{thebibliography}
\end{document}